\documentclass{amsart} %letterpaper
\usepackage{amsfonts,amssymb,amscd,amsmath,enumerate,tikz-cd,url,verbatim, times} %thmtools,thm-restate,
\usepackage[all]{xy}
\usepackage[top=1.2in, bottom=1.2in, left=1.2in, right=1.2in]{geometry}
\usepackage[hidelinks]{hyperref}
\SelectTips{eu}{}
\xyoption{curve}
\CompileMatrices

\newcommand{\po}[3][R]{P^{#1}_{#2}(#3)}

\def\cA{\mathcal A}
\def\cB{\mathcal B}

\def\a{\mathfrak{a}}

\def\m{\mathfrak{m}}

\def\kz{\operatorname{K}}
\def\hz{\operatorname{H}}

\newcommand{\ov}{\overline}
\newcommand{\wt}{\widetilde}

\newcommand{\ges}{{\scriptscriptstyle\geqslant}}
\newcommand{\col}{\colon}
\newcommand{\dd}{\partial}
\newcommand{\del}{\partial}

\newcommand{\fm}{{\mathfrak m}}

\newcommand{\fn}{{\mathfrak n}}

\newcommand{\bd}{\boldsymbol}

\newcommand{\gr}{\operatorname{gr}}

\newcommand{\rank}{\operatorname{rank}}
\newcommand{\HH}{\operatorname{H}}
\newcommand{\hh}{\operatorname{H}}

\newcommand{\Tor}{\operatorname{Tor}}
\newcommand{\Ext}{\operatorname{Ext}}

\newcommand{\Po}{\operatorname{P}}

\renewcommand{\>}{\rangle}

\def\A{\mathcal A}
\def\B{\mathcal B}

\numberwithin{equation}{section}

\theoremstyle{plain}
\newtheorem{theorem}{Theorem}[section]

\newtheorem*{Lemma}{Lemma}
\newtheorem*{Proposition}{Proposition}

\newtheorem*{Main Theorem}{Main Theorem}

\newtheorem{proposition}[theorem]{Proposition}
\newtheorem{lemma}[theorem]{Lemma}
\newtheorem{corollary}[theorem]{Corollary}

{\Alph{theorem}}
{\Alph{theorem}}

\theoremstyle{definition}

\newtheorem{chunk}[theorem]{}

\newtheorem{remark}[theorem]{Remark}

\theoremstyle{remark}

\newenvironment{bfchunk}{\begin{chunk}\textit}{\end{chunk}}
\newtheorem*{Claim1}{Claim 1}
\newtheorem*{Claim2}{Claim 2}

\numberwithin{equation}{theorem}

\title[Algebra structure of Koszul homology]{Detecting Koszulness and related homological properties\\from the algebra structure of Koszul homology}

\author[Croll]{Amanda Croll}
\address{Amanda Croll, Concordia University Irvine, 1530 Concordia, Irvine, CA 92612}
\email{amanda.croll@cui.edu}

\author[Dellaca]{Roger Dellaca} 
\address{Roger Dellaca, Department of Mathematics, University of California, Irvine, 340 Rowland Hall, Irvine CA 92697}
\email{rdellaca@uci.edu}

\author[Gupta]{Anjan Gupta}
\address{Anjan Gupta, Dipartimento di Matematica,
Universit\`{a} degli studi di Genova, Via Dodecaneso 35, 16146, Genova}
\email{agmath@gmail.com}

\author[Hoffmeier]{Justin Hoffmeier}
\address{Justin Hoffmeier, Department of Mathematics and Statistics, Northwest Missouri State University, 800 University Drive, Maryville, MO 64468}
\email{jhoff@nwmissouri.edu}

\author[Mukundan]{Vivek Mukundan} 
\address{Vivek Mukundan, 
Purdue University; current address: School of Mathematics, Tata institute of fundamental research, Mumbai, India
}
\email{vmukunda@purdue.edu}

\author[Rangel Tracy]{Denise Rangel Tracy}
\address{Denise Rangel Tracy, Department of Mathematics, Manhattan College, 4513 Manhattan College Parkway, Riverdale, NY 10471 }
\email{drangeltracy01@manhattan.edu}

\author[\c{S}ega]{Liana M.~\c{S}ega}
\address{Liana M.~\c{S}ega\\ Department of Mathematics and Statistics\\
   University of Missouri\\ \linebreak Kansas City\\ MO 64110\\ U.S.A.}
     \email{segal@umkc.edu}

\author[Sosa]{Gabriel Sosa}
\address{Gabriel Sosa, Department of Mathematics and Statistics, Amherst College, Amherst, MA}
\email{gsosa@amherst.edu}

\author[Thompson]{Peder Thompson}
\address{Peder Thompson\\ Department of Mathematics and Statistics\\ Texas Tech University\\ Broadway and Boston\\ Lubbock, TX 79409}
\email{peder.thompson@ttu.edu}

\date{\today}	% Activate to display a given date or no date
\subjclass[2010]{13D07, 13D02, 16S37}
\thanks{This material is based upon work supported by the National Science Foundation under Grant \#1321794, as part of the  Mathematical  Research Communities 2015 program in Snowbird, Utah, and by a grant from the Simons Foundation (\# 354594, Liana \c Sega). The third author was supported by the Istituto Nazionale di Alta Matematica ``Francesco Severi" fellowship.} 
\thanks{{\em Key words and phrases:}
Koszul homology, Golod homomorphism, Koszul algebra}

\begin{document}
\maketitle
\begin{abstract}
Let $k$ be a field and $R$ a standard graded $k$-algebra. We denote by $\hz^R$ the homology algebra of the Koszul complex on a minimal set of generators of the irrelevant ideal of $R$.  We discuss the relationship between the multiplicative structure of $\hz^R$ and the property that $R$ is a Koszul algebra. More generally, we work in the setting  of local rings and we show that certain conditions on the multiplicative structure of Koszul homology imply strong homological properties, such as existence of certain Golod homomorphisms, leading to explicit computations of Poincar\'e series. As an application, we show that the Poincar\'e series of all finitely generated modules over a stretched Cohen-Macaulay local ring are rational, sharing a common denominator.
\end{abstract}
%\tableofcontents

\section*{Introduction}

Let $(R,\fm,k)$ denote a  local Noetherian ring $R$ with maximal ideal $\fm$ and residue field $k$.  Let $\kz^R$ denote the Koszul complex on a minimal generating set of $\fm$, and let $\hh^R$ denote its homology. The Koszul complex $\kz^R$ can be endowed with the structure of a differential graded (DG) algebra, and is the first step in constructing a DG algebra minimal free resolution of $k$ over $R$ (called a {\it Tate resolution} of $k$ over $R$) through the process of adjoining DG algebra variables. It is thus natural to expect that the properties of the homology algebra $\hh^R$ are related to other homological properties of $R$.  Indeed, it is known that both the Gorenstein and complete intersection properties of $R$ can be characterized in terms of $\hh^R$. 

Certain higher order homology operations on Koszul homology, introduced by Golod \cite{Golod}, can be used to characterize extremality in the growth of the minimal free resolution of $k$ over $R$. If the ring $R$ is Golod, then it has the property that for all finitely generated $R$-modules $M$ the Poincar\'e series $\sum_{i\ge 0}\rank_k(\Tor_i^R(M,k))z^i$ are rational and share a common denominator, see  \cite{GG75}. This property is also satisfied by other large classes of rings, and the recent papers \cite{RS} and \cite{KSV} provide insight into the fact that the multiplicative structure of Koszul homology plays a role in establishing such results. In this paper we further explore how the structure of $\hh^R$ can be used to derive rationality of Poincar\'e series and other homological properties of $R$. In particular, we give special attention in the graded case to the Koszul property. 

 Recall that the Koszul homology of a Golod ring has trivial multiplication, see \cite{Golod}. When $R$ is not Golod, we find it useful to consider conditions on $\hh^R$ that, to some extent, generalize the condition that multiplication is trivial.  We require that cycles living ``deep" enough in $\kz^R$ (i.e. ones that are contained in $\fm^i\kz^R$ for large enough values of $i$)  can be expressed, up to a boundary, in terms of certain cycles that have trivial products among themselves. More precisely, we consider the following conditions on $\kz^R$, depending on integers $t, b, s$: 
\begin{itemize}
\item[$\mathcal Z_{t,b,s}$:] There exists a finite set $Z\subseteq Z(\fm^t\kz^R)$ such that $zz'=0$ for all $z, z'\in Z$ and for every $v \in \m^s\kz^R$ there exists $m\in \mathbb N$ and $z_i\in Z$, $u_i \in Z(\m^{b}\kz^R)$  for each $i$ with $1\le i\le m$, such that 
$v - \sum_{i=1}^m z_iu_i \in B(\m^{s-1}\kz^R)$. 
\end{itemize}
\begin{itemize}
\item[$\mathcal P_{t}$:] There exists $[l]\in \HH_1(\kz^R)$ such that for every $z\in Z(\fm^t\kz^R)$ there exists $z'\in Z(\fm^{t-1}\kz^R)$ such that $z-z'l\in B(\fm^{t-1}\kz^R)$.
\end{itemize}
After setting some ground work in the first two sections, in Sections \ref{Generation_by_special_set_section} and \ref{Generation_by_one_element_section} we prove various  homological implications of the conditions $\mathcal Z_{t,b,s}$ and $\mathcal P_{t}$, under specific conditions on the integers $t, b,s$.  The main results regarding these conditions are Theorems \ref{thm1} and \ref{thm2}. The conclusions of the theorems and their corollaries  are formulated in terms of vanishing of the natural maps $\Tor_*^R(\fm^j,k)\to\Tor_*^R(\fm^i,k)$ induced by the inclusions $\fm^j\subseteq \fm^i$ for certain values of $i$, $j$, as well as identifying Golod homomorphisms, establishing generation of the Yoneda algebra $\Ext_R(k,k)$  in low degreess, and deducing rationality of Poincar\'e series. Our arguments utilize the DG algebra structure of the minimal free resolution of $k$ over $R$. This approach is inspired by, and generalizes, work of Levin and Avramov \cite{GL}, where homological properties of local Gorenstein artinian rings are derived from the fact that the Koszul homology algebra of such a ring is a Poincar\'e algebra. 

For appropriate values of $t$, the property $\mathcal P_{t}$ holds for the class of  compressed Gorenstein artinian local rings discussed in \cite{RS} and also for the class of compressed level local artinian  rings of odd socle degree, see \cite{KSV}. In particular, our results in  Section \ref{Generation_by_one_element_section} can be used to recover the results of \cite{RS} and \cite{KSV} regarding the fact that when the socle degree is different than three, these rings can be obtained as homomorphic images of a hypersurface, via a Golod homomorphism. 

In Section \ref{stretched} we show that the property $\mathcal P_{2}$ is satisfied in the case of stretched artinian rings satisfying $\fm^3\ne 0$ and $\rank_k(\fm/\fm^2)\ne \rank_k(0\colon \fm)$. The class of stretched Cohen-Macaulay local rings  was considered by Sally in \cite{Sally80}, where she proves that the Poincar\'e series of the residue field over such a ring is rational. 
A consequence of our results on generation in the Koszul homology algebra is that the Poincar\'e series of all finitely generated $R$-modules over a stretched Cohen-Macaulay local ring $R$ are rational, sharing a common denominator. Theorem \ref{main-s} also states that the Yoneda algebra $\Ext_R(k,k)$ of a stretched artinian local ring $(R,\fm,k)$ is generated in degree $1$ if $\rank_k(\fm/\fm^2)\ne \rank_k(0\colon \fm)$ and in degrees $1$ and $2$ if  $\rank_k(\fm/\fm^2)= \rank_k(0\colon \fm)$. 

For the remainder of the introduction, assume that $R$ is a standard graded $k$-algebra. Let $\kz^R$ denote the Koszul complex on a set of minimal generators of the irrelevant ideal of $R$ and let $\hz^R$ denote the homology algebra of $\kz^R$. The $k$-algebra $R$ is said to be a {\it Koszul algebra} if the resolution of $k$ over $R$ is linear, that is to say, the differentials in the minimal graded free resolution of $k$ can be represented by matrices of linear forms (see e.g., \cite{S} and \cite{HI}). The algebra $\hz^R$  is bigraded; when writing the bidegree $(i,j)$ of an element, the index  $i$ denotes homological degree and the index $j$ denotes internal degree. The {\it linear strand} of $\hz^R$ is the set of elements  of bidegree $(i,i+1)$,  and the {\it nonlinear strands} are composed of elements of bidegree $(i,i+r)$ with $r>1$. We say that the nonlinear strands of $\hh^R$ are generated by a set $\ov Z\subseteq \hh^R$ if the nonlinear strands are contained in the ideal generated by $\ov Z$ in $\hh^R$. If the nonlinear strands are generated by a subset $\ov Z$ of the linear strand, it follows that $\hh^R$ is generated by the linear strand as a $k$-algebra. 

In Section \ref{graded} we interpret our earlier results in the graded setting, with special attention to the Koszul property. In particular, we obtain the following statements, which  provide new homological criteria for verifying that an algebra is Koszul:  
\begin{enumerate}
\item If the nonlinear strands of $\hz^R$ are generated by one element of bidegree $(1,2)$, then $R$ is absolutely Koszul, hence Koszul. (See \cite{IR} or  Section \ref{graded} regarding absolutely Koszul algebras.)
\item If $R_{\ges 3}=0$ and there exists a set of cycles $Z$ representing elements in the linear strand, with the property that $zz'=0$ for all $z, z'\in Z$ and such that the nonlinear strand of $\hh^R$ is generated by $\ov Z=\{[z]\mid z\in Z\}$, then $R$ is Koszul. 
\end{enumerate}

Recall that in \cite{ACI} and \cite{Boocher-others}, the authors make the point that if $R$ is Koszul, then part of the Koszul homology algebra $\hh^R$ is generated by elements in the linear strand. More precisely, in \cite[Theorem 4.1]{ACI} it is shown that if $R$ is Koszul then $\hh^R_{i,j}=0$ for $j>2i$ and $\hh^R_{i,2i}=(\hh^R_{1,2})^i$ for all $i\ge 0$ and in \cite[Theorem 3.1]{Boocher-others} it is proved that one has  also  $\hh^R_{i,2i-1}=(\hh^R_{1,2})^{i-2}\hh^R_{2,3}$ for all $i\ge 2$.  Section \ref{graded} and Section \ref{examples} provide some further insight into the connections between the fact that $R$ is Koszul and the structure of $\hh^R$. In \ref{MHK} we note that if  $R$ is Koszul then the nonlinear strands of $\hh^R$ are contained in the set of matric Massey products of $\kz^R$. However, generation of $\hz^R$ by the linear strand (which implies that the nonlinear strands of $\hh^R$ are contained in the set of matric Massey products) does not imply that $R$ is Koszul. This can be seen by means of  the example in \ref{ring_not_Koszul}, which relies on a ring from a paper of Roos \cite{Roo16} (The  fact that  the linear strand need not generate $\hz^R$ as a $k$-algebra when $R$ is Koszul is also noted in \cite[Remark 3.2]{Boocher-others}.)  On the other hand, \ref{ring_gen_by_set} describes a Koszul algebra $R$ for which $\hh^R$ has the same bigraded Hilbert series as the homology algebra of the ring of \ref{ring_not_Koszul} and  is also generated  by the linear strand. It turns out that the ring in \ref{ring_gen_by_set}  satisfies the hypothesis of statement (2) above. This observation sheds some light on our effort to understand what distinguishes one homology algebra from the other in the two examples. 

The examples in Section \ref{examples} utilize the Macaulay2 package \texttt{DGAlgebras} written by Frank Moore, which provides an efficient way to verify rings for which statements (1) or (2) hold; using this, we apply our Theorem \ref{graded-thm} to the rings studied in Roos \cite{Roo16}. The last section also contains a concrete example of how our results can be used towards establishing homological properties of the ring and computations of Poincar\'e series, see \ref{4socle}.

Given the evidence that the properties considered in this paper show up in a large variety of situations, we hope that this study will be useful in further explorations of homological properties of local rings. \\

\noindent
{\bf Acknowledgments:}
We are grateful to Jan-Erik Roos and Aldo Conca for useful comments and suggestions.  
 
The work on this paper started during the 2015 Mathematical Research Communities program in Commutative Algebra, under the guidance of Liana \c{S}ega.  The authors would also like to thank the other organizers of this program, Srikanth Iyengar, Karl Schwede, Gregory Smith, and Wenliang Zhang, for their support, and also the AMS staff that coordinated the program. The third author joined the project following conversations during the conference in honor of Craig Huenke held in Ann Arbor in July 2016,  and is thankful for support to travel to this conference from the Department of Mathematics at the University of Michigan and IIT Bombay.

\section{Background}
\label{background_section}
In this section we set notation and provide needed definitions. We recall the definition of a  small homomorphism and provide  some preliminary results centered on this concept. 

\begin{chunk}\label{prelim}
Let $(R,\fm,k)$ be a local ring and $M$ a finite (meaning finitely generated) $R$-module. Fix a minimal generating set of $\fm$ and let $\kz^R$ denote the Koszul complex on this set. Let $\hz^R$ denote the homology algebra of $\kz^R$.  The complex $\kz^R$ has a natural structure of a graded commutative algebra, and this structure is inherited by $\hz^R$.  We denote by $\kz^M$ the Koszul complex $\kz^R\otimes_RM$. 

 The {\it Poincar\'e series} $\Po_M^R(z)$ of $M$ is defined as $$\Po^R_M(z)=\sum_{i\ge 0}\rank_k(\Tor_i^R(M,k))z^i.$$

If $\phi: (R,\fm,k)\to(S,\fn,k)$ is a surjective homomorphism of local rings then the following coefficientwise inequality holds $$\Po_M^S(z)\preccurlyeq\frac{\Po_M^R(z)}{1-z(\Po_S^R(z)-1)}\,.$$ If equality holds for $M=k$ then we say that $\phi$ is a {\it Golod homomorphism}.

The homomorphism $\phi$ induces maps $$\Ext_{\phi}^i(k,k):\Ext_S^i(k,k)\to\Ext_R^i(k,k).$$ If $\Ext_{\phi}^*(k,k)$ is surjective then we say $\phi$ is {\it small}. Recall that if $\phi$ is Golod then $\phi$ is small (cf.\! Avramov \cite[3.5]{Av}). 
\end{chunk}

When $R$ is artinian, specific conditions formulated in terms of the concepts above allow for an explicit computation of the series $\Po^R_k(z)$. 

\begin{lemma}\label{lem_poincare_rational}
Let $(R,\fm,k)$ be an artinian local ring with $\fm^{s+1}=0$. Let $n$ denote the minimal number of generators of $\fm$ and let $a$ denote the dimension of $\fm^s$.  

If the canonical projection $R\to R/\fm^s$ is small and the ring $R/\fm^s$ is Golod, then the Poincar\'e series of $k$ over $R$ is rational, satisfying the formula: 
\begin{equation}\label{Golod-form}
\Po^R_k(z)=\frac{(1+z)^n}{1-z(\hh^{R/\fm^s}(z)-1)+az^2(1+z)^n}
\end{equation}
where $\hh^{R/\fm^s}(z)$ stands for the Hilbert series (which is in this case a polynomial of degree $n$) of the Koszul homology algebra $\hh^{R/\fm^s}$. 
\end{lemma}

\begin{proof}
As $\fm^s\cong k^a$, we have  $\Po^R_{R/\fm^s}(z)=az \Po_k^R(z)+1$. 

Since $R/\fm^s$ is Golod and $R\rightarrow R/\fm^s$ is small, $R\rightarrow R/\fm^s$ is a Golod homomorphism by \cite[6.7]{S}. Therefore, we have 
$$
\Po_k^R(z)=\Po_k^{R/\fm^s}(z)\left(1-z(\Po^R_{R/\fm^s}(z)-1)\right)=\Po_k^{R/\fm^s}(z)\left(1-az^2\Po_k^R(z)\right)\,.
$$
By rearranging, we have that
\begin{equation}\label{2background}
\Po_k^R(z)=\frac{\Po_k^{R/\fm^s}(z)}{1+az^2\Po_k^{R/\fm^s}(z)}\,.
\end{equation}

Finally, since $R/\fm^s$ is Golod, we have that
\begin{equation}\label{3background}
\Po_k^{R/\fm^s}(z)=\frac{(1+z)^n}{1-z(\hh^{R/\fm^s}(z)-1)}\,.
\end{equation}
 The conclusion follows from \eqref{2background} and \eqref{3background}. 
\end{proof}

In order to apply the lemma, we need to verify that the canonical projection $R\to R/\fm^s$ is Golod or small. In \cite{GL}, Levin and Avramov prove that this homomorphism is Golod (thus small) whenever the artinian ring $R$ is Gorenstein. Their proof relies on the fact that $\hh^R$ is a Poincar\'e algebra when $R$ is Gorenstein.  This leads us to believe that, more generally, the structure of the  Koszul homology $\HH^R$ can be used to understand the homological properties of the canonical projection $R\to R/\fm^s$. The next lemma shows that matric Massey products play a role. 

\begin{chunk}\label{r}
Let $\varphi\colon R\to S$ be a homomorphism of local rings.  The {\it usual products} of $\HH^R$ are understood to be the set of products $\HH_{\ges 1}^R\cdot \HH_{\ges 1}^R$. We denote by $MH(\kz^R)$ the set of {\it matric Massey products} of $\HH_{\ges 1}^R$, as defined in \cite{May}. This set of higher order homology operations is a submodule of $\HH_{\ges 1}^R$ and contains the usual products; see also \cite[(1.4.1)]{Avr86} for a more concise definition. 

The induced map $\HH(\kz^{\varphi})\colon \HH(\kz^R)\to \HH(\kz^S)$ satisfies $\HH(\kz^{\varphi})(MH(\kz^R))\subseteq MH(\kz^S)$. By \cite[4.6]{Av}, if $\varphi$ is small then the induced homomorphism
$$
\HH_{\ge 1}(\kz^R)/MH(\kz^R)\to \HH_{\ge 1}(\kz^S)/MH(\kz^S)
$$
is injective. From here we derive immediately the following statement: 

\begin{Lemma}
Let $(R,\fm,k)$ be a local ring and let $i\ge 0$. Consider the conditions: 
\begin{enumerate}[\quad\rm(1)]
\item The canonical projection $R\to R/\fm^i$ is small; 
\item $\HH_{\ge 1}(\fm^i \kz^R)\subseteq MH(\kz^R)$.
\end{enumerate}
Then {\rm (1)} implies {\rm (2)}. \qed
\end{Lemma}
\end{chunk}

Example \ref{ring_not_Koszul} in Section \ref{examples} shows that the implication (2)$\implies$(1) does not hold when $i=2$. Ideally, one would like to replace condition (2) with a stronger one, that is equivalent to (1). While such a condition is not yet known, we identify in Sections \ref{Generation_by_special_set_section} and \ref{Generation_by_one_element_section} two conditions on Koszul homology that imply (1), for certain values of $i$.

\section{A property of the Tate resolution}
\label{tate_resolution_section}
The purpose of this section is to record in Proposition \ref{tate-prop} a general property of the Tate resolution. This result will be used later, in the proof of one of the main theorems. We start with a description of the Tate resolution, and we invite the reader to consult \cite{Avr98} for more details, in particular for the definition of a DG algebra. We then build the  ingredients of the proof of the proposition by means of a couple of lemmas.  

\begin{bfchunk}{Adjunction of variables.}
Let $B$ be a DG algebra over $R$ and suppose $z$ is a cycle in $B$.  We embed $B$ into a DG algebra $B'=B\<y\>$ by freely adjoining a variable $y$ such that $\dd(y)=z$ as follows: 

If $|z|$ is even, the variable $y$ such that $\dd(y)=z$ is called an {\em exterior variable} and satisfies $y^2=0$. Denote by $k\<y\>$ the exterior algebra over $k$ of a free $k$-module on a generator of degree $|z|+1$.  The differential on $B\<y\>=B\otimes_k k\<y\>$  is given by
$$\dd(b_0+b_1y)=\dd(b_0)+\dd(b_1)y+(-1)^{|b_1|}b_1z.$$

 If $|z|>0$ is odd, $y$ is a {\em divided powers variable}.  The $k$-algebra $k\<y\>$ on a divided powers variable $y$ is the free $k$-module with basis $\{y^{(i)}:|y^{(i)}|=i|y|\}_{i\geq 0}$ and multiplication table
$$y^{(i)}y^{(j)}={i+j \choose i} y^{(i+j)}\text{, for $i,j\geq 0$.}$$
We set $y^{(1)}=y$, $y^{(0)}=1$, and $y^{(i)}=0$ for $i<0$.  Forgetting the differentials, $B\<y\>=B\otimes_k k\<y\>$.  If $z\in B$ is a cycle of positive odd degree, then
$$\dd\left(\sum_i b_i y^{(i)}\right)=\sum_i \dd(b_i)y^{(i)}+\sum_i (-1)^{|b_i|}b_izy^{(i-1)}$$
is a differential on $B\<y\>$ that extends that of $B$ and satisfies the Leibniz rule.

The notation $B\<y_1, \dots, y_n\>$ stands for the DG algebra obtained by repeated adjunction of variables as above. 
\end{bfchunk}

\begin{bfchunk}{The Tate resolution.}\label{tate_res}
Let $x_1,...,x_n$ be a minimal generating set for $\m$ and $\kz^R$ the Koszul complex on $x_1,...,x_n$.  Note that we can interpret $\kz^R$ as the DG algebra
$$\kz^R=R\<T_1,...,T_n\>,$$
where $T_i$ are degree 1 exterior variables (these exterior variables in degree 1 will be referred to as {\em Koszul variables}) with $\dd(T_i)=x_i$.  One can continue to ``kill" homology by adjoining variables to $\kz^R$,  following the construction in \cite[6.3.1]{Avr98}. The resulting DG algebra $\mathcal A$ is a minimal free resolution of $k$ over $R$,  often referred to as the {\it Tate resolution} of $k$ over $R$. 

Forgetting differentials, $\A$ is a free $R$-module: see  \cite[Remark 6.2.1]{Avr98} for a description of the basis in terms of the variables adjoined. In particular, one can see that $\A$ is also a free $\kz^R$-module. 
\end{bfchunk}

\begin{bfchunk}{Notation.}
We write $\kz$ for $\kz^R$ when the ring $R$ is understood.
Since  $\cA$ is a free algebra over $\kz$, we consider a homogeneous $\kz$-basis of $\cA$. For each $j$, let $\chi_{i,j}$ with $1\le i\le q_j$  denote the elements of homological degree $j$ in this basis. If $z\in \cA_p$, then we write $z$ in terms of this basis as 
\begin{equation}
\label{basis}
z=\sum_{j=0}^p\sum_{i=1}^{q_j}z_{i,j}\chi_{i,p-j}
\end{equation}
with $z_{i,j}\in \kz_{j}$ for each $j$. 

As usual, $Z(\cA)$ denotes the set of cycles of $\cA$ and $B(\cA)$ denotes  the set of boundaries. 
\end{bfchunk}

The following lemma provides the inductive step for our key lemma, Lemma \ref{z}, below. 
\begin{lemma}
\label{induction-step}
Let  $z\in Z_p(\cA)$ and write it as in \eqref{basis}. Let $a$ be an integer with $0\le a\le p$.  
Assume  $z_{i,j}\in \fm^t \kz$ for all $j$ with $0\le j\le a-1$ and all $i$ with $1\le i\le q_j$. 
Then  
$$\dd(z_{i,a})\in \fm^{t+1}\kz$$
for all $i$ with $1\le i\le q_a\,.$
\end{lemma}

\begin{proof}
 Since $z\in Z(\cA)$, we have $\dd(z)=0$. On the other hand we can compute $\dd(z)$ from \eqref{basis},  using the Leibniz rule; this yields:
\begin{equation}
\label{1-tate}
0=\sum _{j=0}^p\sum_{i=1}^{q_j}(-1)^{j}z_{i,j}\dd(\chi_{i,p-j})+\dd(z_{i,j})\chi_{i,p-j}.
\end{equation}
Let $i$ be such that $1\le i\le q_a$. We express all terms in the right hand side of \eqref{1-tate} in terms of the $\kz$-basis of $\cA$, and collect the terms to compute the coefficient of each basis element in the sum. We see that the coefficient of $\chi_{i,p-a}$ in this sum is 
$$
\dd(z_{i,a})+ \sum _{j=0}^{a-1}\sum_{i'=1}^{q_j}z_{i',j}w_{i',j}\,\quad \mbox { with } w_{i',j}\in \fm \kz_{a-j}\,.
$$
The coefficients $w_{i',j}$ come from expressing $\dd(\chi_{i,p-j})$ in terms of the $\kz$-basis, for $j\le a-1$. In particular, $w_{i',j}\in \fm \kz$ since $\cA$ is minimal.  (Note that if $j\ge a$ then $z_{i,j}\dd(\chi_{i,p-j})$ does not have any contribution to the coefficient of $\chi_{i,p-a}$, for degree reasons.)
These coefficients must equal $0$, hence
$$
\dd(z_{i,a})=-\sum _{j=0}^{a-1}\sum_{i'=1}^{q_j}z_{i',j}w_{i',j}\in (\fm^t\kz)(\fm \kz)\subseteq \fm^{t+1}\kz
$$
for all $i$ with $1\le i\le q_a$. 
\end{proof}

\begin{chunk}
\label{v}
Let $\widehat R$ denote the completion of $R$ with respect to $\fm$. We may write $\widehat R=Q/I$, with $(Q,\fn,k)$ a regular local ring and $I\subseteq \fn^2$; this presentation is called a {\it minimal Cohen presentation}. We set 
\begin{align*}\label{vR}
v(R)=\max\{j\mid I\subseteq \fn^j\}.
\end{align*}
As noted in \cite{HS}, this integer is independent of the choice of the minimal Cohen presentation. 
\end{chunk}

\begin{remark}
\label{I} If $v(R)\ge t+1$ then the map 
\begin{equation}
\label{original}
\HH_i(\kz^R/\fm^{t+1}\kz^R)\to \HH_i(\kz^R/\fm^{t}\kz^R)
\end{equation}
induced by the canonical homomorphism $\kz^R/\fm^{t+1}\kz^R\to \kz^R/\fm^{t}\kz^R$ is zero for all $i\ge 1$. 
In particular, we have:  If $\dd(z)\in \fm^{t+1}\kz^R$, then $z\in B(\kz^R)+\fm^{t}\kz^R$ for all $z\in \kz^R_{\geqslant 1}$. 

Indeed, to justify this statement it suffices to assume that $R$ is complete, with $R=Q/I$ as above. We can write $\kz^R=\kz^Q\otimes_QR$, where $\kz^Q$ is the Koszul complex on a minimal generating set of $\fn$ (obtained by lifting the minimal generating set picked for $\fm$). Since $I\subseteq \fn^{t+1}$ by assumption, we can make the identifications 
$\kz^R/\fm^{t+1}\kz^R=\kz^Q/\fn^{t+1}\kz^Q$ and $\kz^R/\fm^{t}\kz^R=\kz^Q/\fn^{t}\kz^Q$. The map in \eqref{original} can then be identified with the induced map 
$$\HH_i(\kz^Q/\fn^{t+1}\kz^Q)\to \HH_i(\kz^Q/\fn^{t}\kz^Q)$$
which is zero for all $i\ge 1$ because the induced map $\HH_i(\fn^{t+1}\kz^Q)\to \HH_i(\fn^t\kz^Q)$ is zero for all $i\ge 0$, since $Q $ is regular (for example, see \cite[Theorem 3.3]{S}). 
\end{remark}

We are now prepared to prove a key lemma; a reformulation of this will yield Proposition \ref{tate-prop} below.
\begin{lemma}
\label{z}
Suppose $(R,\fm,k)$ is  a local ring, $\kz$ is the Koszul complex on a minimal generating set $x_1,...,x_n$ of $\fm$, and $\cA$ is the Tate resolution of $k$. Let $t\ge 1$ be an integer such that $v(R)\ge t+1$. 
If $x\in \cA$, then there exists $y\in \kz_1\cA$ such that $\dd(x-y)\in \fm^t\kz_1\cA$.  
\end{lemma}

\begin{proof}
Let $x\in \cA_{p+1}$. If $p=0$, we may take $y=x$ and the result follows trivially.  Now assume $p\geq 1$.  
Since $\cA$ is minimal, we have $\dd(x)=\sum x_ig_i$ with $x_i\in \fm$ and $g_i\in \cA_p$.  Choose $A_i\in \kz_1$ such that $\dd(A_i)=x_i$ and we have 
\begin{align*}
\dd\left(x-\sum A_ig_i\right)&=\dd(x)-\sum \dd(A_ig_i)\\
&=\sum x_ig_i- \sum \dd(A_i)g_i+\sum A_i\dd(g_i)\\
&=\sum A_i\dd(g_i).
\end{align*}
Set $y'=\sum  A_ig_i$. Then $y'\in \kz_1\cA$ and $\dd(x-y')\in \kz_1\cA$.

Apply Lemma \ref{induction-step} with $z=\dd(x-y')$ and $a=1$ (noting that $z_{i,0}=0$ for all $i$ with $1\le i\le q_0$, so the hypothesis is satisfied).  This yields $\dd(z_{i,1})\in\fm^{t+1}\kz$, and then Remark \ref{I} shows 
$$
z_{i,1}=\dd(e_{i,1})+f_{i,1}
$$
for some $e_{i,1}\in \kz_2$ and $f_{i,1}\in \fm^t\kz$. Consequently: 
$$
\dd(x-y')=\sum_{i=1}^{q_{1}} (\dd(e_{i,1})+f_{i,1})\chi_{i,p-1}+V,\quad\text{with}\quad V\in \kz_2\cA.
$$
Now take $
y_1=\sum_{i=1}^{q_{1}} e_{i,1}\chi_{i,p-1}$
and we have: 
\begin{align*}
\dd(x-y'-y_1)&=\sum_{i=1}^{q_{1}} (\dd(e_{i,1})+f_{i,1})\chi_{i,p-1}+V-\sum_{i=1}^{q_{1}} \dd(e_{i,1}\chi_{i,p-1})\\ 
&=\sum_{i=1}^{q_{1}} f_{i,1}\chi_{i,p-1}+\left(V+\sum_{i=1}^{q_{1}}e_{i,1}\dd(\chi_{i,p-1})\right)\\
&=\sum_{i=1}^{q_{1}} f_{i,1}\chi_{i,p-1}+V_1, \quad\text{with } V_1\in \kz_2\cA.
\end{align*}
Set $m = \min\{p,n\}$, and let us assume inductively, for $a-1<m$, that we constructed $y_1, y_2, \dots, y_{a-1}\in \kz_1\cA$ such that 
$$
\dd(x-y'-y_1-\dots -y_{a-1})=\sum_{j=1}^{a-1}\sum_{i=1}^{q_j}f_{i,j}\chi_{i,p-j}+V_{a-1},
$$
with $f_{i,j}\in \fm^t\kz$ and $V_{a-1}\in \kz_{a}\cA$. Applying again Lemma \ref{induction-step} and Remark \ref{I}, with $z=\dd(x-y'-y_1-\dots -y_{a-1})$ we have that  
$$z_{i,a}=\dd(e_{i,a})+f_{i,a}
$$
with $e_{i,a}\in \kz_{a+1}$ and $f_{i,a}\in \fm^t\kz_a$. Consequently, we can write
$$
\dd(x-y'-y_1-\dots-y_{a-1})=\sum_{j=1}^{a-1}\sum_{i=1}^{q_j}f_{i,j}\chi_{i,p-j}+\sum_{i=1}^{q_{a}} (\dd(e_{i,a})+f_{i,a})\chi_{i,p-a}+V,
$$
with $V\in \kz_{a+1}\cA$. 
Now take 
$
y_a=\sum_{i=1}^{q_{a}} e_{i,a}\chi_{i,p-a}
$
and, as above, we get: 
$$
\dd(x-y'-y_1-\dots -y_{a})=\sum_{j=1}^{a}\sum_{i=1}^{q_{j}}f_{i,j}\chi_{i,p-j}+V_{a},
$$
with $V_a\in \kz_{a+1}\cA$. Note that $V_a=0$ when $a\ge m$, by degree reasons (if $p<n$) and since $K_{>n}=0$.

Set $y=y'+y_1+\dots +y_m$. Then the cycle $z=\dd(x-y)$ satisfies the conclusions of our statement. 
\end{proof}

We can now prove the useful decomposition property of the Tate resolution advertised above, which was  inspired by the work in \cite{GL}.
\begin{proposition}
\label{tate-prop}
Let $t\ge 1$ be an integer such that $v(R)\ge t+1$. Denote by $\cA'$ the DG subalgebra of $\cA$ given by 
$$
 \cA'=\{x\in \cA\mid \dd(x)\in \fm^t\kz_1\cA\}.
$$
Then $\cA$ is generated by $\cA'$ as a $\kz$-algebra, that is: 
$
\cA=\cA'+\kz_1\cA'+\kz_2\cA'+\dots +\kz_n\cA'=\kz\cA'.
$
\end{proposition}

\begin{proof}
For $x\in \cA$, Lemma \ref{z} provides an element $y\in \kz_1\cA$ such that $x-y\in \cA'$.
A reformulation of Lemma \ref{z} therefore gives that $\cA=\cA'+\kz_1\cA$. Applying this fact repeatedly, and noting that $\kz_{n+1}=0$, we  get: 
\begin{align*}
\cA&= \cA'+ \kz_1\cA=\cA'+\kz_1(\cA'+\kz_1\cA)=\cA'+\kz_1\cA'+\kz_2\cA=\cdots \\
&= \cA'+\kz_1\cA'+\kz_2\cA'+\cdots +\kz_n(\cA'+\kz_1\cA)\\
&=\cA'+\kz_1\cA'+\kz_2\cA'+\cdots +\kz_n\cA'\\
&=\kz\cA'.
\end{align*}
\end{proof}

\section{Generation by a special set}
\label{Generation_by_special_set_section}
We continue with the notation of the previous sections for the Koszul complex and the Tate resolution of a local ring $R$. In this section, we prove one of the main theorems, Theorem \ref{thm1} below, and we point out its applications. In particular, these applications include a computation of the Poincar\'e series $\Po_k^R(z)$ and conditions under which the map $R\to R/\m^s$ is Golod. 

Recall that the invariant $v(R)$ was introduced in \ref{v}. 

\begin{theorem}
\label{thm1}
Let $(R,\m,k)$ be a local ring and let $s$ be an integer such that $\m^{s+1}=0$.    
Let $t$ and $b$ be integers such that $s-t\le b\le s-1$ and $v(R)\ge t+1\ge 2$, and assume that the following condition holds:
\begin{itemize}
\item[$\mathcal Z_{t,b,s}$:] There exists a finite set $Z\subseteq Z(\fm^t\kz^R)$ such that $zz'=0$ for all $z, z'\in Z$
and for every $v \in \m^s\kz^R$ there exists $m\in \mathbb N$ and $z_i\in Z$, $u_i \in Z(\m^{b}\kz^R)$  for each $i$ with $1\le i\le m$, such that 
$v - \sum_{i=1}^m z_iu_i \in B(\m^{s-1}\kz^R)$.  
\end{itemize}
The maps $\Tor_i^R(\m^s,k)\to \Tor_i^R(\m^{b},k)$ induced by the inclusion $\fm^s\subseteq \fm^b$ are then zero for all $i\ge 0$.
\end{theorem}

We postpone the proof of the theorem in order to give some corollaries. We will use below, and also in the next section, the following result of Rossi and \c Sega.

\begin{chunk}
\cite[Lem.~1.2]{RS}\label{RS1.2} Let  
$\kappa:(P,\mathfrak p,k)\to (R,\mathfrak m,k)$ be 
 a  
surjective homomorphism of local rings.  
If there exists a positive integer $a$ such that{\rm:}
\begin{enumerate}[\rm(a)]
\item the map $\Tor^ P_i(R,k)\to\Tor^ P_i(R/\mathfrak m^a,k)$ induced by the canonical quotient map $R\to R/\mathfrak m^a$ is zero for all positive $i$, and
\item the map $\Tor^P_i(\mathfrak m^{2a},k)\to \Tor^P_i(\mathfrak m^a,k)$ induced by the inclusion $\mathfrak m^{2a}\subseteq \mathfrak m^a$ is zero for all non-negative integers $i$,\end{enumerate}
then $\kappa$ is a Golod homomorphism.
\end{chunk}

\begin{corollary}
Under the hypotheses of Theorem {\rm \ref{thm1}}, if  $2b\ge  s$  then the homomorphism $R\to R/\m^s$ is Golod.
\end{corollary}
\begin{proof}
We apply \ref{RS1.2} to the  natural projection $\kappa\col R\to R/\fm^s$. Since $2b \ge s$, we have $\m^{2b}/\m^s = 0$, so the map $\Tor_i^R(\m^{2b}/\m^s,k) \to \Tor_i^R(\m^{b}/\m^s,k)$ induced by the inclusion $\m^{2b}/\m^s \subseteq  \m^b/\m^s$ is zero for all $i \ge 0$. By Theorem \ref{thm1}, we have that the induced maps $\Tor_i^R(R/\m^s,k) \to \Tor_i^R(R/\m^b,k)$ are also zero for all $i>0$. Thus, the conditions of  \ref{RS1.2} are satisfied and the result follows.
\end{proof}

\begin{remark}\label{rem}
Let $R=Q/I$ be a minimal Cohen presentation of $R$, with $(Q,\fn,k)$ a regular local ring. As first noted by L\"ofwall \cite{Lo}, the ring $R$ is Golod whenever there exists an integer $t$ such that 
\begin{equation}
\label{inclusions}
\fn^{2t}\subseteq  I\subseteq \fn^{t+1}
\end{equation}

Assume $s$ and $t$ are integers such that $\fm^{s+1}=0$ and $v(R)\ge t+1$. If $s<2t$, then the inclusions in \eqref{inclusions} hold, and  it follows that  $R$ is Golod.  If $s=2t$ then $R$ is not necessarily Golod, but it follows that the quotient ring $R/\fm^s=Q/(I+\fn^s)$ is Golod. 
\end{remark}

\begin{corollary}
\label{formula}
Assume the hypothesis of Theorem {\rm \ref{thm1}} is satisfied. If $s=2t$ and $b=t$, then the hypotheses of Lemma {\rm \ref {lem_poincare_rational}} are satisfied, and thus $\Po_k^R(z)$ satisfies the formula {\rm\eqref{Golod-form}}. 
\end{corollary}
\begin{proof}
The homomorphism $R\to R/\fm^s$ is Golod by Theorem \ref{thm1}, and thus small (see \ref{prelim}). The ring $R/\fm^s$ is Golod by Remark \ref{rem}.
\end{proof}

Concrete examples when these results can be applied will be given in Section \ref{examples}. 

\begin{proof}[Proof of Theorem {\rm\ref{thm1}}] Let $|Z|$ denote the cardinality of $Z$. 
Let $\{z_1,...,z_{|Z|}\}$ be the cycles in $Z$ and let $\mathcal{I}$ denote the set of all finite ordered lists of elements in $\{1, \ldots, |Z|\}$, including the empty set. 

Let $I\in \mathcal I$.  If $I=(i)$ has length $1$, we set $I^{-}=\emptyset$. If $I=(i_1, \dots, i_r)$ has length $r\ge 2$, we set $I^{-}=(i_1, \dots, i_{r-1})$. 
We now define for each $I\in \mathcal I$ an element $y_I$ such that 
\begin{enumerate}
\item $y_I=1\in \cA_0$, if $I=\emptyset$;
\item $\dd(y_I)=z_{i_r}y_{I^-}$, if $I=(i_1, \dots, i_r)$ with $r\ge 1$. 
\end{enumerate}
The details of constructing these elements are as follows. If $I=\emptyset$, we choose $y_I$ as in (1).  If $r=1$ and $I = (i)$, we can choose $y_I \in \cA$ such that $\del(y_I) = z_i$, since $z_i \in Z$ is a cycle. Assuming that $r\ge 2$ and the elements $y_I$ have been defined for all $I\in \mathcal I$ of length $r-1$, we can construct elements $y_I$  satisfying (2) for $I=(i_1, \dots, i_r)$  by noting that $z_{i_r}y_{I^{-}}$ is a cycle, so such a $y_I$ exists. Indeed, we have 
$$\dd(z_{i_r}y_{I^{-}})=(-1)^{|z_{i_r}|} z_{i_r}\dd(y_{I^-})=(-1)^{|z_{i_r}|} z_{i_r}z_{i_{r-1}}y_{(I^-)^-}=0,$$
where the last equality is due to the hypothesis on the set $Z$. 

We identify the map $\Tor_i^R(\m^s,k)\to \Tor_i^R(\m^{b},k)$ with the map $H_i(\m^s\cA)\to H_i(\m^b\cA)$.  
In order to show this map is trivial for $i\geq 0$, we let $x\in \fm^s\cA_c$ for some $c$ and must show that $x\in \dd(\fm^b\cA)$.

Set $Y= \{y_I \in \cA | I \in \mathcal{I}\}$. For $i, j \geq 0,$ define
$$\cA(i,j):=\left(\m^s\kz_iY_j\cA'\right)\cap \cA_c,$$
where $Y_j$ is the set of elements of $Y$ of degree $j$.

\medskip

\noindent {\it Claim}: For $i,j\geq 0$,
\begin{equation}
\label{claim}
\cA(i,j) \subseteq \del(\m^{b}\cA) + \sum_{\stackrel{p+q= i+j}{p>i}} \cA(p,q)+\sum_{p+q> i+j}\cA(p,q)
\end{equation}
where only finitely many terms in this sum are not zero, for degree reasons. 

\medskip

\noindent {\it Proof of Claim}: 
To prove the inclusion, it suffices to consider elements of $\cA(i,j)$ of the form $vy_Ia'$ for some $v \in \m^s\kz_i$, $a'\in \cA'$ and $y_I \in Y_j$,with $I=(i_1, \ldots, i_r) \in \mathcal{I}$ or $I=\emptyset$; in the last case, we set $r=0$.   (Note that every element of $\cA(i,j)$ can be written as a sum of elements of this form.)

By assumption  there exist $ z_{\iota, v} \in Z$ and $u_{\iota,v}\in Z(\fm^{b}\kz)$ with $1\le \iota\le m$ such that 
$$ v- \sum_{\iota=1}^m z_{\iota, v}u_{\iota,v}\in \del(\m^{s-1}\kz_{i+1}).$$ 

We need to show thus that $(\sum_{\iota=1}^m z_{\iota, v}u_{\iota,v}+w)y_I\cA'$ is contained in the right hand side of \eqref{claim}, for all  $w\in \del(\m^{s-1}\kz_{i+1})$, $ z_{\iota, v} \in Z$ and $u_{\iota,v} \in Z(\fm^{b}\kz)$. Note that it suffices to prove this statement when $m=1$. We assume thus that $v- z_{i_{r+1}}u\in \del(\m^{s-1}\kz_{i+1})$ for some $z_{i_{r+1}}\in Z$ and $u\in Z(\fm^{b}\kz)$ and we  show that $vy_I\cA'$ is contained in the right hand side of \eqref{claim}. 

In what follows, we will simplify the notation when $I$ has length $1$ and we will write $y_i=y_{(i)}$.  
Define $I^+=(i_1,\dots, i_r, i_{r+1})$. Recall that we defined a set $I^{-}$ for any nonempty $I\in \mathcal I$. Although we do not define a set $\emptyset^{-}$, we agree to set $y_{I^{-}}=z_{i_r}=y_{i_r}=0$ when $I=\emptyset$. With this convention, note that the formula $\dd(y_I)=z_{i_r}y_{I^-}$  also holds when $I=\emptyset$, so we will not treat this case separately. 

Note that
\begin{align*}
& |u| =i-|z_{i_{r+1}}|= i+1-|y_{i_{r+1}}|,\\
&|y_{I^-}| = j - |z_{i_r}| - 1 = j - |y_{i_r}|  \quad\text{when}\quad  I\ne \emptyset, \text{ and} \\
&|y_{I^+}| = j + |z_{i_{r+1}}| + 1 = j + |y_{i_{r+1}}|.
\end{align*}

We now have 
\begin{eqnarray*}
vy_I\cA' &\subseteq& \left(\del(\m^{s-1}\kz_{i+1})+z_{i_{r+1}}u\right)y_I\cA' \\
&=&\del(\m^{s-1}\kz_{i+1})y_I\cA' +u \del(y_{I^+})\cA' \\
& =&\del(\m^{s-1}\kz_{i+1}y_I\cA') +u\del(y_{I^+})\cA' +\m^{s-1}\kz_{i+1}z_{i_r}y_{I^-}\cA'+\m^{s-1}\kz_{i+1}y_I\del(\cA') \\
& \subseteq & \del(\m^{s-1}\cA) +\del(u y_{I^+}\cA') +uy_{I^+}\del(\cA')+\m^{s-1}\kz_{i+1}z_{i_r}y_{I^-}\cA'+ \m^{s-1}\kz_{i+1}y_I\del(\cA'),
\end{eqnarray*}
where in line 2 we used the formula $z_{i_{r+1}}y_I=\del(y_{I^+})$, in line 3 we used the Leibniz rule and the formula  $\del(y_I)=z_{i_r}y_{I^{-}}$, and in line 4 we used again the Leibniz rule and  the fact that $u$ is a cycle.

Consider the terms from the last line in the previous display. Since $u \in Z(\m^{b}\kz)$ and $b\le s-1$, we have $\del(\m^{s-1}\cA) +\del(u y_{I^+}\cA') \subseteq \del(\m^{b}\cA)$. Additionally, since $\dd(\cA')\subseteq \fm^t \kz_1\cA$ by the definition of $\cA'$, we have $uy_{I^+}\del(\cA') \subseteq \m^{b+t}\kz_{|u|+1} y_{I^+}\cA$ and $\m^{s-1}\kz_{i+1}y_I\del(\cA') \subseteq \m^{s-1+t} \kz_{i+2}y_I\cA$. Furthermore, since $z_{i_r} = \del(y_{i_r})$, we have $\m^{s-1}\kz_{i+1}z_{i_r}y_{I^-}\cA' \subseteq \m^s\kz_{i+|y_{i_r}|} y_{I^-}\cA'$. Thus
$$vy_I\cA' \subseteq \del(\m^{b}\cA) +\m^{b+t}\kz_{|u|+1} y_{I^+}\cA+ \m^s\kz_{i+|y_{i_r}|} y_{I^-}\cA'+\m^{s-1+t} \kz_{i+2}y_I\cA.$$
Using Proposition \ref{tate-prop} and the facts $b +t \geq s$ and $t \geq 1$, we have
$$
vy_I\cA' \subseteq \del(\m^{b}\cA) +\m^{s}\kz_{|u|+1} y_{I^+}\cA' + \m^s\kz_{i+|y_{i_r}|} y_{I^-}\cA'+\m^{s} \kz_{i+2}y_I\cA',
$$
with the provison that if $I=\emptyset$, the third term on the right of this inclusion is $0$ by convention. 
Finally, since $|u|= i+1-|y_{i_{r+1}}|$, we conclude
$$vy_I\cA' \subseteq \del(\m^{b}\cA) +\sum_{i'\ge i+2-|y_{i_{r+1}}|} \cA(i',j+|y_{i_{r+1}}|) + \cA(i+|y_{i_r}|,j-|y_{i_r}|) + \sum_{i'\ge i+1}\cA(i'+1,j),$$
with the caveat that we remove the term  $\cA(i+|y_{i_r}|,j-|y_{i_r}|)$ from the right-hand sum when $I=\emptyset$. 
In the display above, the second and last terms of the right-hand side are sums of the form $\cA(p,q)$ with $p+q>i+j$ and the third term is of the form  $\cA(p,q)$ with $p+q=i+j$ and  $p>i$. The Claim is thus proved.

To finish the proof of the theorem, consider an order on the set $M=\{(i,j)|i,j \geq 0 \}$ as follows:  Order the elements by $i+j$ first, then by $i$ as a tiebreak.  In other words: 
$$(i,j)>(i',j') \iff i+j>i'+j' \quad\text{or}\quad (i+j=i'+j' \quad \text{and}\quad i>i').
$$

Recall that $x\in \fm^s\cA_c$ and we need to show $x\in \dd(\fm^b\cA)$. Using Proposition \ref{tate-prop}, we know that $x\in \sum_{i=0}^n \fm^s \kz_i\cA'$.
Thus if $M_0=\{(i,0)\mid 0\le i\le n\}$, then 
$$x \in \sum_{(i,j) \in M_0}\cA(i,j)\,.$$ Using the Claim, we see that there exists a finite set $M_1$ and an element
 $$x_1\in \sum_{(i,j) \in M_1}\cA(i,j)$$ such that $x-x_1\in \dd(\fm^b\cA)$, and such that the smallest element of $M_1$ is strictly larger (in the order described above) than the smallest element of $M_0$. 

Applying the Claim again, this time using $x_1$, we see that there exists a finite set $M_2$ and an element $x_2$ such that $$x_2\in \sum_{(i,j) \in M_2}\cA(i,j)$$ such that $x_1-x_2\in \dd(\fm^b\cA)$, hence $x-x_2\in \dd(\fm^b\cA)$, and  such that the smallest element of $M_2$ is strictly larger (in the order described above) than the smallest element of $M_1$. A repeated use of the argument  ensures the construction of elements $x_a$ and sets $M_a$ such that for each integer $a$ the smallest element of $M_a$ is greater than the smallest element of $M_{a-1}$ and such that $x-x_a\in \dd(\fm^b\cA)$. When $a$ is sufficiently large (so that $p+q>c$ for all $(p,q)\in M_a$) one sees that, for degree reasons, we have  $\cA(p,q)=0$ for all $(p,q)\in M_a$,  since the elements of $\cA(p,q)$ are in $\mathcal A_c$, with $c$ fixed.  We conclude that  $x_a=0$ for $a$ sufficiently large, hence $x\in \dd(\fm^b\cA)$. 
\end{proof}

\section{Generation by one element}\label{Generation_by_one_element_section}
We now turn our focus to the case where the Koszul homology algebra is generated by a single element. Namely, here we are concerned with rings that satisfy the following condition, depending on integers $t$ and $r$: 
\begin{itemize}
\item[$\mathcal P_{t,r}$:] There exists $[l]\in \HH_r(\kz)$ such that for every $z\in Z(\fm^t\kz)$ there exists $z'\in Z(\fm^{t-1}\kz)$ such that $z-z'l\in B(\fm^{t-1}\kz)$.
\end{itemize}
Note that this condition is independent of the choice of the representative $l$ of the class $[l]\in \HH_r(\kz)$.

\begin{remark}\label{not-minimal}
The condition $\mathcal P_{t,r}$ is particularly strong when the cycle $l$ is not a minimal generator of $Z_r(\kz)$. In this case, we can choose  $l\in \fm Z_r(\kz)$. Since $\kz$ is constructed using a minimal generating set for $\fm$, we have that any cycle in $Z_r(\kz)$ is also in $\fm \kz$ and thus in $Z_r(\fm \kz)$. Let $i\ge 0$. 
We have thus $lz'\in \fm Z_{i} (\fm^t \kz)$ for all $z'\in  Z_{i-r}(\fm^{t-1}\kz)$.  The hypothesis that  $\mathcal P_{t,r}$ holds then implies 
$$Z_i(\fm^t \kz)\subseteq \fm Z_i(\fm^t \kz)+B_i(\fm^{t-1}\kz)$$
for all $i$, and so by Nakayama's Lemma, we have $Z_i(\fm^t\kz)\subseteq B_i(\fm^{t-1}\kz)$, hence the induced map $\HH_i(\fm^t\kz)\to \HH_i(\fm^{t-1}\kz)$ is zero for all $i\geq 0$.

When $l$ is part of a minimal generating set for $Z_r(\kz)$ and $r$ is odd,  we see below that a similar statement can be deduced, only that the complex $\kz$ needs to be replaced with a larger complex. 
\end{remark}

We denote by $\cB$ the following DG algebra
$$\mathcal B=\kz\langle y\mid \dd(y)=l\rangle\,.$$

\begin{theorem}
\label{thm2}
Let $(R,\fm,k)$ be a local ring, let $t\ge 2$ be an integer, and $r\ge 1$ be an odd integer. If $\mathcal P_{t,r}$ holds, then the map
$$
\HH_i(\fm^{t}\mathcal B)\to \HH_i(\fm^{t-1}\mathcal B)
$$
induced by the inclusion $\fm^t\hookrightarrow\fm^{t-1}$ is zero for all  $i\ge 0$.
\end{theorem}

\begin{proof}
We will first prove two claims.
 
\medskip

\noindent {\it Claim 1}: For a cycle $z\in K$, if $z=z'+w$ with $z'\in \fm^t\kz_py^{(q)}$ and $w\in\fm^t \kz_{p+1}\cB$, then $z'\in  Z_{p}(\fm^t\kz)y^{(q)}$.

\medskip

\noindent {\it Proof of Claim 1}: Set $|z|=c$, hence $p+(r+1)q=c$. Write $z'=vy^{(q)}$ with $v\in \fm^t\kz_p$. We have 
\begin{equation}
\label{1}
0=\dd(z)=\dd(vy^{(q)})+\dd(w)=\dd(v)y^{(q)}+(-1)^{p}vl y^{(q-1)}+\dd(w)\,.
\end{equation}
Since $w\in \fm^t\kz_{p+1}\cB$ and $|w|=c$ we can write $\dd(w)=\sum_{i,j} k_iy^{(j)}$ with $k_i\in \fm^t\kz_i$, where the sum is taken over non-negative $i,j$ with $i+(r+1)j=c-1$, and $i\ge p$. Since $p+(r+1)q=c$, we must have $j<q$ for all such $j$. 
Since $1, y,y^{(2)},y^{(3)},\dots$ is a basis of $\cB$ over $\kz$, we conclude from \eqref{1} that $\dd(v)=0$, and hence $v\in Z_p(\fm^t\kz)$. 

\medskip

\noindent {\it Claim 2}:  If  $z'\in Z_{p}(\fm^t\kz)y^{(q)}$, then
$$
z'\in \dd(\fm^{t-1}\cB)+\fm^t \kz_{p+r+1}y^{(q-1)}.
$$

\medskip

\noindent {\it Proof of Claim 2}:  Let $z'=vy^{(q)}$ with $v \in Z_p(\fm^t \kz)$.  The hypothesis of the theorem gives  $v-v'l\in \dd(\fm^{t-1}\kz_{p+1})$ for some  $v'\in Z_{p-r}(\m^{t-1}\kz)$, and so $z'- v'ly^{(q)}\in \dd(\fm^{t-1}\kz_{p+1})y^{(q)}$.  An application of the Leibniz rule yields
$$z'- v'ly^{(q)}=z'-v'\dd(y^{(q+1)})=z'-(-1)^{|v'|}\dd(v'y^{(q+1)}),$$
and therefore
\begin{align*}
z'\in &  \dd(\fm^{t-1}\cB)+\dd(\fm^{t-1}\kz_{p+1})y^{(q)}\\
\subseteq & \dd(\fm^{t-1}\cB)+\dd(\fm^{t-1} \kz_{p+1}y^{(q)})+\fm^{t-1}\kz_{p+1}\dd(y^{(q)}),\text{ by the Leibniz rule,} \\ 
\subseteq &\dd(\fm^{t-1}\cB)+\fm^t \kz_{p+1+r}y^{(q-1)},
\end{align*}
which verifies Claim 2. 

Now suppose $z\in \cB$ represents a nontrivial class in $\hh_c(\m^t \cB)$ and let $p$ be the largest integer such that $z\in \fm^t\kz_p\cB$. Write $z=\sum k_iy^{(j)}$, where the sum is taken over integers $i,j$ with $i+(r+1)j=c$ and $i\ge p$, and $k_i\in \fm^t\kz_i$. Set $z'=k_py^{(\frac{c-p}{r+1})}$ and $w=z-z'$. Then $z=z'+w$ and the hypotheses of Claim 1 are satisfied.  
Putting together Claim 1 and Claim 2, we have:   
\begin{equation}
\label{2}
z\in \dd(\fm^{t-1}\cB)+z_1,\qquad \text{for some $z_1\in \fm^t\kz_{p+1}\cB$}\,.
\end{equation}
As $z_1$ is also a cycle, we may repeat the argument for $z_1$, and so on.  Inductively, we obtain $z\in \dd(\fm^{t-1}\cB)+ \fm^t\kz_i\cB$ for all $i> p$. Since $\kz_i=0$ for $i\gg0$, we get $z\in \dd(\fm^{t-1}\cB)$. 
\end{proof}

We recall below a result of Levin \cite[Lemma 2]{Lev78}. 
\begin{chunk}
\label{Levin}
Let $\mathcal F$ be a differential graded $R$-algebra and $\mathcal E$ a differential graded $\mathcal F$-module which is free as a graded $\mathcal F$-module (i.\,e. forgetting differentials) and such that $\dd(\mathcal E)\subseteq\fm\mathcal E$. Let $M$, $N$ be $R$-modules such that $\fm M\subseteq N\subseteq M$, and such that the canonical map $N\otimes _R\mathcal E\to M\otimes_R\mathcal E$ is injective. If the induced homomorphism $\HH(N\otimes_R \mathcal F)\to \HH(M\otimes_R \mathcal F)$ is zero, then so is the induced homomorphism $\HH(N\otimes_R \mathcal E)\to \HH(M\otimes_R \mathcal E)$. 
\end{chunk}

\begin{corollary}\label{tor}
If $\mathcal P_{t,1}$ holds, then the following hold: 
\begin{enumerate}[\quad\rm(1)]
\item The map
$$
\Tor_i^R(\fm^t,k)\to \Tor_i^R(\fm^{t-1},k)
$$
induced by the inclusion $\fm^t\subseteq \fm^{t-1}$ is zero for all $i\ge 0$.
\item  If $v(R)\ge t$, then the algebra $\Ext_R(k,k)$ is generated in degrees $1$ and $2$. 
\item If $t=2$, then the algebra $\Ext_R(k,k)$ is generated in degree $1$. 
\end{enumerate}
\end{corollary}

\begin{proof} 
If $l\notin \fm Z_1(\kz)$, we  can then construct a minimal Tate resolution $\mathcal A$ of $k$ over $R$ by starting with the Koszul complex K, then adjoining a variable $y$ with $\dd(y)=l$, and then the rest of the needed variables as described in \ref{tate_res}. The description of the basis of $\mathcal A$ in \cite[Remark 6.2.1]{Avr98} shows that, forgetting differentials, $\mathcal A$ is free over the algebra $\mathcal B=\kz\<y\mid \dd(y)=l\>$. Using \ref{Levin} and Theorem \ref{thm2}, we conclude that the induced map 
$$
\HH_i(\fm^t\mathcal A)\to \HH_i(\fm^{t-1}\mathcal A)
$$
is zero for all $i\ge 0$, and this yields the desired conclusion. 

If  $l\in \fm Z_1(\kz)$, then we showed in  Remark \ref{not-minimal} that the induced map $\HH_i(\fm^t\kz)\to \HH_i(\fm^{t-1}\kz)$ is zero for all $i\geq 0$, and the conclusion follows again by applying Levin's result in \ref{Levin}. 

In view of (1), part (2) follows then from \cite{HS} and part (3) from \cite[Corollary 1]{Ro}. 
\end{proof}

\begin{corollary}
\label{P-R}
Let $(Q,\fn,k)$ be a regular local ring, and $I$ an ideal with $I\subseteq \fn^2$, and consider the local ring $(R,\fm,k)$ defined by $R=Q/I$. Let $h\in I\smallsetminus \fn I$ and let $L\in \kz^Q$ such that $\dd(L)=h$. Let $l$ denote the image of $L$ in $\kz=\kz^Q\otimes_QR$. Assume that $\mathcal P_{t,1}$ holds, with $l$  as above.  

If  $P=Q/(h)$, then the induced map $\Tor_i^P(\fm^t,k)\to \Tor_i^P(\fm^{t-1},k)$ is  zero for all $i\ge 0$. 

Furthermore, let $a$ be an integer such that $I\subseteq \fn^{a+1}$. If $h\in I\smallsetminus \fn^{a+2}$ and $a+1\le t\le 2a$, then the canonical projection $P\to R$ is a Golod homomorphism. 
\end{corollary}

\begin{remark}
The existence of a surjective Golod homomorphism from a complete intersection ring to the local ring $R$ is a rather remarkable property: Using a result of Levin recorded in  \cite[Proposition 5.18]{AKM88}, this property allows one to conclude that the Poincar\'e series of all finitely generated $R$-modules are rational, sharing a common denominator. 
\end{remark}

\begin{proof}
The hypothesis implies that  $h$ is part of a minimal generating set of $I$. 

Set $\B=\kz\<y\mid \dd(y)=l\>$ and $\B'=\kz^Q\<y\mid \dd(y)=L\>$.  Note that $\B=(\B'\otimes_QP)\otimes_PR$ and $\B'\otimes_QP$ is a minimal free resolution of $k$ over $P$.  We can identify thus the induced map   $\Tor_i^P(\fm^t,k)\to \Tor_i^P(\fm^{t-1},k)$ with the induced map  $\HH_i(\fm^t\cB)\to \HH_i(\fm^{t-1}\cB)$, and the latter is zero, by the theorem. 
Under the additional hypotheses in the last paragraph of the statement, it also follows that the induced map $\Tor_i^P(\fm^{2a},k)\to \Tor_i^P(\fm^{a},k)$ is zero for all $i\ge 0$,  since $t\le 2a$ and $t-1\ge a$. Since $I\subseteq \fn^{a+1}$, \cite[Lemma 1.4]{RS} gives that the induced map $\Tor_i^P(R,k)\to \Tor_i^P(R/\fm^{a},k)$ is zero for all positive $i$, and then \ref{RS1.2} gives the final conclusion.
\end{proof}

Compressed Gorenstein artinian local rings, and, more generally, compressed level artinian local rings are defined in terms of an extremal condition  involving the length, embedding dimension, and the socle of the ring. We refer to \cite{RS} and \cite{KSV} for the precise definitions. Such rings can be viewed as being ``generic", in a sense explained in more detail  in \cite[Theorem 3.1]{KSV}, for example. 

\begin{remark}
Let $(R, \fm, k)$ be a compressed Gorenstein local ring of socle degree $s\ne 3$ and assume $k$ is infinite. Set $t=s+2-v(R)$ and $a=v(R)-1$. 
When $s\ne 3$, the inequalities $a+1\le t\le 2a$ follow from general properties of compressed Gorenstein rings, and more precisely from the fact that  $s=2v(R)-1$ or $s=2v(R)-2$, as noted in \cite{RS}. 

The  proof of \cite[Proposition 4.6]{RS} shows that $Z_n(\kz)\subseteq lZ_{n-1}(\fm^{t-1}\kz)$ for some cycle $l\in \kz_1 $, and \cite[Lemma 4.4]{RS} shows that the induced map $\HH_i(\fm^t\kz)\to \HH_i(\fm^{t-1}\kz)$ is zero for all $i<n$, where $n$ denores the minimal number of generators of $\fm$.  It follows that the ring $R$ satisfies  $\mathcal P_{t,1}$. We can then construct $h\in I$ such that $\dd(L)=h$, where $L$ is a preimage of $l$ in $\kz^Q$.  With this data, the hypotheses of Corollary \ref{P-R} are satisfied, hence we recover the main structural result in \cite{RS} stating that $R$  is a homomorphic image of a hypersurface via a Golod homomorphism.
\end{remark}

\begin{remark}
The condition $\mathcal P_{t,1}$ is also satisfied for compressed level artinian local rings of socle degree $s=2t-1$, with $s\ne 3$. This can be seen from the proof of \cite[Lemma 6.3]{KSV} and \cite[Lemma 4.4]{KSV}. The conclusion that such rings are homomorphic images of a hypersurface via a Golod homomorphism is established there using the results of \cite{RS}. Corollary \ref{P-R} recovers the same conclusion, as well. 
\end{remark}
 
Another class of rings for which the condition $\mathcal P_{t,1}$ holds is discussed in the next section.

\section{Stretched Cohen-Macaulay local rings}
\label{stretched}

Let $(R,\fm,k)$ be an artinian local ring. We set 
$$v=\rank_k(\fm/\fm^2); \qquad e=\text{length}(R); \qquad r=\rank_k(0\colon \fm); \qquad h=e-v\,.$$ 
Note that  $\fm^{h+1}=0$. The ring $R$ is said to be {\it stretched}  if $h$ is the least integer $i$ such that $\fm^{i+1}=0$. Assume further that $R$ is not a field. Then $R$ is stretched if and only if $\fm^2$ is principal. Sally \cite{Sally80} computed the series $\Po_k^R(t)$  for a stretched artinian local ring $R$ as follows: 
\begin{equation}
\label{P-stretched}
\Po_k^R(t)=\begin{cases} {1}/{(1-vt)} &  \text{if $r=v$};\\
{1}/{(1-vt+t^2)}& \text{if $r\ne v$}.
\end{cases}
\end{equation}
In this section we show that finitely generated modules over $R$ have rational Poincar\'e series as well, sharing a common denominator. The main result is Theorem \ref{main-s}; its proof involves an application of the results in Section \ref{Generation_by_one_element_section}.

\begin{bfchunk}{Structure of stretched artinian local rings.}
\label{structure}
Assume that $(R,\fm,k)$ is  a stretched artinian local ring, not a field. We further set 
$$p=v-r\quad \text{and}\quad q=r-1\,.$$

Assume that $h\ge 3$.  (The case $h\le 2$ has been treated in \cite{AIS08} and will be recalled later.)  As described in \cite{Sally80}, we choose elements $t, z_1, \dots, z_{p}, w_1, \dots, w_{q}$ forming a minimal generating system of $\fm$ such that the following hold: 
\begin{enumerate}
\item $\fm^i=(t^i)$ for all $i\ge 2$;  
\item The elements $w_1, \dots, w_{r-1}, t^{h}$ form a basis of $(0\colon \fm)$; in particular, $tw_j=0$ and $z_iw_j=0$ for all $i$, $j$ with $1\le i\le p$ and $1\le j\le q$; 
\item $tz_i=0$ for all $i$ with $1\le i\le p$;
\item $z_iz_j=a_{ij}t^h$, with $a_{ij}=0$ or $a_{ij}$ a unit of $R$, for all  $i$, $j$ with $1\le i, j\le p$. 
\end{enumerate}
(Note that there are no elements $w_i$ if $r=1$ and there are no elements $z_i$ if $r=v$.)

If $r\ne v$, let $\ov{a_{ij}}$ denote the image of $a_{ij}$ in $k=R/\fm$ and note that the matrix $(\ov{a_{ij}})$ is invertible. Indeed, if this matrix is not invertible, then there exists an element $z=\sum_{i=1}^pb_iz_i$, with $b_i\in R$ such that $b_i$ is a unit for at least one index $i$, and such that $zz_j=0$ for all $j$ for all $1\le j\le p$.  This implies that $z\in (0\colon \fm)$, hence $z\in (w_1, \dots, w_{r-1}, t^h)$, contradicting the fact that $t, z_1, \dots, z_{p}, w_1, \dots, w_{q}$ is a minimal generating set of $\fm$. 

The fact that the matrix $(\ov{a_{ij}})$ is invertible also implies that for every $i$  with $1\le i\le p$ there exists an element $y_i\in (z_1,\dots, z_p)$ such that 
\begin{equation}
\label{zt1}
z_iy_i=t^h\quad\text{and}\quad z_jy_i=0\quad\text{for all $j$ with $i\ne j$, $1\le j\le p$}\,.
\end{equation} 
Since $y_i\in (z_1, \dots, z_p)$ we also have
\begin{equation}
\label{zt2}
ty_i=0=w_jy_i\quad \text{for all $i$, $j$ with $1\le i\le p$ and $1\le j\le q$}\,.
\end{equation}
\end{bfchunk}

\begin{bfchunk}{Structure of Koszul homology.} 
\label{structure-K} Let $R$ be as in \ref{structure}. We  consider the Koszul complex  $K^R$ on the set  $\{w_1, \dots, w_q, z_1, \dots, z_p, t\}$. As a DG algebra it can be described as the complex $$K^R=R\langle W_1, \dots, W_q, Z_1, \dots, Z_p, T\rangle$$ with $\dd(W_i)= w_i$, $\dd(Z_i)=z_i$ and $\dd(T)= t$.

 In what follows, we will consider products of some of the variables $W_i$ or $Z_i$,  with the index $i$  ranging over certain sets $I$. We adopt the convention that the  product is equal to $1$ if $I=\emptyset$; for example $Z_1\cdots Z_p=1$ if $p=0$. We set
$$
\mathcal W=\{W_{j_1}W_{j_2}\cdots W_{j_i}\mid 1\le j_1<j_2<\dots<j_i\le q, \,0\le  i\le q\}
$$
where $W_{j_1}W_{j_2}\cdots W_{j_i}=1$ when $i=0$. 

\begin{Lemma}
Let $s\ge 1$. Every cycle of $Z_s(\fm^2K^R)$ has the form 
$$\sum a_{W}t^hTZ_1\cdots Z_pW+V,\qquad \text{
with $V\in B_s(\fm K^R)$}\quad\text{and}\quad a_W\in R,
$$
where  the sum ranges over all $W\in \mathcal W$ with $|W|=s-p-1$. 
\end{Lemma}

\begin{proof}
If $L=Z_{i_1}Z_{i_2}\cdots Z_{i_m}W_{j_1}W_{j_2}\cdots W_{j_n}$ with $1\le i_1<i_2<\cdots<i_m\le p$ and $1\le j_1<j_2<\cdots<j_n\le q$,  then the Leibniz rule and the fact that $ t w_j=0=t z_i$ for all  $1\le i\le p$ and $1\le j\le q$ imply that $t\dd(L)=0$. Consequently, another application of the Leibniz rule gives
$$
 t^2L=\dd( tTL)\,.
$$
Since $\fm^2=(t^2)$, it suffices to consider cycles of the form  $z=\sum_L a_Lt^2TL$ with $L$ ranging over all possible products $L$ as above (if $p=0=q$, then we only have one term where $L=1$). 
Taking differentials, we have 
$$
0=\dd(z)=\sum a_Lt^3L 
$$
where the last equality follows again from the Leibniz rule and the fact that $t\dd(L)=0$, as noted above.  Since the products $L$ are linearly independent over $R$, we have $a_Lt^3=0$ for all $L$, and hence $a_L=t^{h-2}b_L+c_L$, with $b_L\in R$ and $c_L\in  (w_1, \dots, w_q, z_1, \dots, z_p)$.  We have thus $z=\sum_L b_Lt^hTL$. 

Consider a cycle $t^hTL$, with $L$ as above. Set $Z=Z_{i_1}Z_{i_2}\cdots Z_{i_m}$ and $W=W_{j_1}W_{j_2}\cdots W_{j_n}$, so that $t^hTL=t^hTZW$. We will show that $t^hTZW\in \dd(\fm K^R)$ whenever $Z\ne Z_1Z_2\cdots Z_p$.  

Indeed, assume  $Z\ne Z_1Z_2\cdots Z_p$. In particular, we have $p\ne 0$. Without loss of generality, we may assume $Z_1$ is not a factor in $Z$. Then for any such $Z$ and any $W$ as above we have: 
$$
t^h TZW=z_1y_1TZW=\dd(y_1Z_1TZW)\\
$$
where the element $y_1$ is as defined above, and the last equality follows from the Leibniz rule, in view of the relations in \eqref{zt1} and \eqref{zt2}. 
\end{proof}
\end{bfchunk}

\begin{bfchunk}{Generation of Koszul homology.}
\label{generation}
We continue with the notation and hypotheses of \ref{structure} and \ref{structure-K}. In addition, we assume $r\ne v$, hence $p\ne 0$.  We set:  
$$
 \mathcal J_1=\{(i,j)\mid 1\le i\le j\le p, a_{ij}\ne 0\}\quad \text{and}\quad \mathcal J_2=\{(i,j)\mid 1\le i\le j\le p, a_{ij}= 0\}\,.
$$
Consider the following element in $K_1^R$: 
\begin{align}
\label{F}F=\sum_{1\le  i< j\le q}&\alpha_{ij} w_iW_j+\sum_{1\le j\le q, 1\le i\le p}\beta_{ij}w_jZ_i+\sum_{1\le i\le q}\gamma_i w_iT+\\
&+\sum_{1\le i\le p}\delta_itZ_i+\sum_{(i,j)\in\mathcal J_1} \eta_{ij}( t^{h-1}T-a_{ij}^{-1}z_iZ_j)+\sum_{(i,j)\in\mathcal J_2}\theta_{ij}z_iZ_j\nonumber
\end{align}
where  $\alpha_{ij}$, $\beta_{ij}$, $\gamma_i$, $\delta_i$, $\eta_{ij}$, $\theta_{ij}\in  R$. Recall that the elements $a_{ij}$ were introduced in \ref{structure}. The relations in \ref{structure} yield that $F$ is a cycle.

\begin{Lemma}
Assume in addition that at least one of the coefficients $\delta_i$, $\eta_{ij}$, or $\theta_{ij}$ is a unit.
If $z$ is a cycle in  $\fm^2K^R$, then $z=AF+A'$ for some  $A\in Z(\fm K^R)$ and $A'\in B(\fm K^R)$. In other words, condition $\mathcal P_{2,1}$ holds with $l=F$. 
\end{Lemma}
\begin{proof}

Using Lemma \ref{structure-K}, we see that it suffices to consider cycles of the form $t^hTZ_1\cdots Z_pW$, with $W\in\mathcal W$. Now consider the following element, which can be seen to be a cycle in $\fm K^R$: 
\begin{align*}
C&=\sum_{1\le i<j\le p}d_{ij}\left (y_iTZ_1\cdots \widehat Z_j\cdots Z_pW+(-1)^{j-i}y_jTZ_1\cdots \widehat Z_i\cdots Z_pW\right )+\\&+\sum_{1\le i\le p}d_{ii}(y_iTZ_1\cdots \widehat Z_i\cdots Z_pW)+\sum_{1\le i\le p}b_{i}\left (t^{h-1}TZ_1\cdots \widehat Z_i\cdots Z_pW+(-1)^iy_iZ_1\cdots Z_pW\right )\end{align*}
with $d_{ij}, b_i\in R$, where the notation $\widehat Z_j$ indicates that the element $Z_j$ is missing from the product. Set 
$$\alpha=\sum_{1\le i\le p}(-1)^ib_i\delta_i-\sum_{(i,j)\in\mathcal J_1} (-1)^jd_{ij}\eta_{ij}a_{ij}^{-1}+\sum_{(i,j)\in\mathcal J_2}(-1)^jd_{ij} \theta_{ij}\,.$$ 
Then we have
$$
FC=\alpha(t^hTZ_1\cdots Z_pW).
$$
Since at least one of the coefficients $\delta_i$, $\eta_{ij}$, or $\theta_{ij}$   is a unit, we can choose the coefficients $d_{ij}$ and $b_i$ such that $\alpha$ is a unit. Then, if we take $A=\alpha^{-1}C$, we have $t^hTZ_1\cdots Z_pW=AF$. 
\end{proof} 
\end{bfchunk}

\begin{theorem} 
\label{main-s}
Let $(R,\fm,k)$ be a stretched artinian local ring with minimal Cohen presentation $R=Q/I$, where $(Q,\fn,k)$ is a regular local ring and $I\subseteq \fn^2$. Assume $R$ is not a field.   

Set $v=\rank_k(\fm/\fm^2)$ and $r=\rank_k(0\colon \fm)$.  The following hold: 
\begin{enumerate}[\quad\rm(1)]
\item If $r\ne v$, then $I\not\subseteq  \fn(I:\fn)$ and we have: 
\begin{enumerate}[\rm(a)]
\item If $\fm^3\ne 0$, then the induced homomorphism $Q/(f)\to R$ is Golod for all $f\in I\smallsetminus \fn(I:\fn)$.
\item The algebra $\Ext_R(k,k)$ is generated by its elements of degree $1$.
\item $(1+t)^v(1-vt+t^2)\Po_M^R(t)\in\mathbb Z[t]$ for all finitely generated $R$-modules $M$. 
\end{enumerate}
\item If $r=v$, then
\begin{enumerate}[\rm(a)]
\item $R$ is a Golod ring. 
\item The algebra $\Ext_R(k,k)$ is generated by its elements of degree $1$ and $2$.
\item $(1-vt)\Po_M^R(t)\in\mathbb Z[t]$ for all finitely generated $R$-modules $M$. 
\end{enumerate}
\end{enumerate}
\end{theorem}

\begin{remark}
\label{m^3}
If $R$ is a stretched artinian local ring with $\fm^3=0$ and $r\ne v$, then the hypotheses of \cite[Theorem 4.1]{AIS08} apply. In view of \cite[Theorem 1.4]{AIS08}, we have that, after a faithfully flat extension, there exists a regular local ring $(Q,\fn,k)$,  an element $u\in \fn$, and a Golod surjective homomorphism $Q/(u^2)\to R$. In this case, as noted in \cite{AIS08}, the ring $R$ is Koszul, in the sense that the associated graded ring with respect to $\fm$ is a Koszul algebra, implying in particular part (1b) of the theorem.  In this case,  one has $(1-vt+t^2)\Po_M^R(t)\in\mathbb Z[t]$ for all finitely generated $R$-modules $M$ \! \cite[Theorem 1.1]{AIS08}, and in particular (1c) holds. 
\end{remark}

\begin{proof}
 Choose $\wt w_i$, $\wt t$, $\wt z_i$   preimages of the elements $w_i$, $t$, $z_i$ in $Q$. Let $K^R$ be the Koszul complex described earlier, and let $K^Q$ be the Koszul complex over $Q$ on the set $\{\wt w_1, \dots, \wt w_q, \wt z_1, \dots, \wt z_p, \wt t\}$, and note that $K^R=K^Q\otimes_QR$, hence we can  identify $K^R$ with $K^Q/IK^Q$. 

(1)  Assume $r\ne v$ and $\fm^3\ne 0$.  In this case, the ideal $I$ is generated by the elements 
\begin{gather*}
\wt w_j\wt w_l,\quad\, \wt  w_j\wt z_i,\quad\,  \wt w_j\wt t,\quad\,  \wt z_i\wt t \quad \text{with $1\le j\le l\le q$ and $1\le i\le p$}\,;
\\ \text{$\wt t^h-\wt a_{ij}^{-1}\wt z_i\wt z_j$ with $(i,j)\in \mathcal J_1$}; \quad\, \text{$\wt z_i\wt z_j$ with $(i,j)\in \mathcal J_2$}\,.
\end{gather*}
We also have
$$
(I\colon \fn)=(\wt w_1, \dots, \wt w_q, \wt t^h)+I\quad \text{and}\quad  \fn(I\colon \fn)=\fn (\wt w_1, \dots, \wt w_q)+\fn I.
$$
(Note that, if $R$ is Gorenstein, then $q=0$ and $\fn(I\colon \fn)=\fn I$. ) In particular, note that $I\ne \fn(I\colon \fn)$.

Let $f\in  I\smallsetminus \fn (I\colon \fn)$. Then we can choose elements  $\alpha_{ij}$, $\beta_{ij}$, $\gamma_i$, $\delta_i$, $\eta_{ij}$, $\theta_{ij}\in  R$ in \eqref{F} such that at least one of the elements $\delta_i$, $\eta_{ij}$, or $\theta_{ij}$ is a unit, and such that we can lift the element $F$ in \eqref{F} to an element $L$ in $K^Q$ with  $\dd(L)=f$. 
Then Lemma \ref{generation} shows that the ring $R$ satisfies $\mathcal P_{2,1}$, with $l=F$. Part (1a) follows then from Corollary \ref{P-R} and (1b) follows from Corollary \ref{tor}.  Furthermore, the existence of a surjective Golod homomorphism onto $R$ from a hypersurface ring, together with  formula \eqref{P-stretched} and  a result of Levin, see \cite[Proposition 5.18]{AKM88},  prove (1c). 

 Remark \ref{m^3} explains the statements (1b) and (1c) when $\fm^3=0$. 

(2) Assume now that $r=v$, hence $p=0$. We see that the ideal $I$ is then generated by the elements
$$
\wt w_j\wt w_l, \quad\,  \wt w_j\wt t, \,\quad\, \wt t^{h+1} \quad \text{with $1\le j\le l\le q$ }\,.
$$
We set $A=R/(w_1, \dots, w_q)$ and $B=R/(t)$. Note that $(A, \mathfrak a, k)$ is a local ring with $A\cong Q'/(\tau^{h+1})$, where $Q'$ is the regular local ring $Q'=Q/(\wt w_1, \dots, \wt w_q)$ and $\tau$ is the image of $\wt t$ in $Q'$; in particular, $A$ is a hypersurface. Also, note that $(B, \mathfrak b, k)$ is a local ring whose maximal ideal satisfies $\mathfrak b^2=0$. Then $R$ is isomorphic to the fiber product $A\times_kB$. Since $A$ and $B$ are both Golod rings, \cite[Th\'eor\`eme 4.1]{Lescot} implies that $R$ is a Golod ring, establishing (2a).  Also, by \cite{Moore}, the algebra $\Ext_R(k,k)$ is then the coproduct of the algebras $\Ext_A(k,k)$ and $\Ext_B(k,k)$; see {\it loc.\! cit.\!} for the definition of the coproduct.  Since  $\mathfrak b^2=0$, the Yoneda algebra $\Ext_B(k,k)$ is generated by its elements of degree $1$. The Yoneda algebra of the hypersurface $A$ is generated in degrees $1$ and $2$. It follows that the coproduct is generated in degrees $1$ and $2$, establishing (2b).

By a result of Ghione and Gulliksen \cite{GG75}  and formula \ref{P-stretched}, the fact that $R$ is Golod implies that $\po Mt$ is rational, with denominator $(1+t)^v(1-vt)$. The more precise denominator $1-vt$  in (2c) requires additional discussion.  To compute $\po Mt$ we use a method employed in the proof of \cite[Corollary 4.4]{AIS08}. 
By \cite[Rem.~3]{DK} one has $\Omega^R_2(M)=K\oplus L$, where $K$ is
an $A$-module and $L$ is a $B$-module.  Using further a formula in \cite[Thm.~2]{DK}, we have
   \begin{align}
\label{fiber}
\po Mt-\beta^R_0(M)-\beta^R_1(M)\cdot t
&=\po Kt\cdot t^2+\po Lt\cdot t^2
   \\
&=\frac{\po[A]{K}t\cdot\po[B]kt+\po[A]kt\cdot\po[B]Lt}
{\po[A]kt+\po[B]kt-\po[A]kt\cdot\po[B]kt}\cdot t^2\nonumber.
   \end{align}
Since $A$ is a hypersurface, we have  that $\po[A]kt=1/(1-t)$, and the Poincar\'e series of every finitely generated $A$-module can be written as a rational function with denominator $1-t$, since every $A$-module $M$ has an eventually periodic resolution. Since $\mathfrak b^2=0$ and  $\rank_k(\mathfrak b/\mathfrak b^2)=v-1$, we have that  $\po[B]kt=1/(1-(v-1)t)$ and the Poincar\'e series of every finitely generated $B$-module can be written as a rational function with denominator $1-(v-1)t$. 
Plugging these formulas into \eqref{fiber} we obtain that $\po Mt$ can be written as a fraction with denominator $1-vt$, establishing (2c). 
\end{proof}

If $(R,\fm,k)$ is a $d$-dimensional local Cohen-Macaulay ring of multiplicity $e$, then $R$ is said to be {\it stretched} if there exists a minimal reduction $\bd x=x_1, \dots, x_d$ of $\fm$ (that is, there exist $d$ elements $x_1, \dots, x_d$ of $\fm$ such that $\fm^{r+1}=(x_1, \dots, x_d)\fm^r$ for some non-negative integer $r$) such that $R/(\bd x)$ is stretched. Standard arguments allow us to reduce computations of Poincar\'e series over $R$ to computations over the stretched artinian local ring  $R/(x_1, \dots, x_d)$, and we obtain:

\begin{corollary}
Let $(R,\fm)$ be a  $d$-dimensional stretched local Cohen-Macaulay ring of type $r$.  Set $\text{mult}(R)=e$ and $\rank_k(\fm/\fm^2)=v$. There exists then a polynomial $d_R(t)\in \mathbb Z[t]$ such that $d_R(t)\Po_M^R(t)\in\mathbb Z[t]$ for all finitely generated $R$-modules $M$, where
$$
d_R(t)=\begin{cases}1-(v-d)t &  \text{if $r=v-d$};\\
(1+t)^{v-d}(1-(v-d)t+t^2) & \text{if $r\ne v-d$}.
\end{cases}
$$
\end{corollary}

\section{Graded rings and Koszul algebras}
\label{graded}  

We give here graded versions of the main statements. While our results have been stated so far for local rings, they can be stated similarly in the case that $R=Q/I$, with $Q=k[x_1, \dots, x_n]$ with $k$ a field and $I$ a homogeneous ideal, and with $\fm$ denoting the irrelevant ideal $(x_1, \dots, x_n)$. Furthermore, in the graded case, the hypotheses can be formulated into more suggestive language, as we shall point out below, and in particular we obtain the applications to the study of the Koszul property mentioned in the introduction. 

Let $R$ be as above, and let $\kz$ denote the Koszul complex on the images of the variables.  Let $\hh$ denote the homology. Note that $\kz$ and  $\hh$ are bigraded algebras. When we say that an element of $\kz$ or $\hh$ has bidegree $(i,j)$, the entry $i$ denotes homological degree and the entry $j$ denotes internal degree.  If $[z]\in \hh_{i,j}$ is nonzero  we set 
$$d(z)=j-i.$$
  When looking at the Betti table of the resolution of $R$ over $Q$ given by Macaulay2 \cite{M2}, which can be interpreted as also describing the graded Hilbert series of $\hh$, the information $d(z)=r$ indicates that the element $[z]$ lies in the $r$th line (strand) of the table.

With this terminology, we can restate conditions $\mathcal Z_{t,b,s}$ and $\mathcal P_{t,r}$ as follows:
\medskip

\begin{enumerate}
\item[$\mathcal P_{t,r}$:]There exists $[l]\in \hz_r$ such that for every $[z]\in \hz$ with $d(z)\ge t$, there exists $[z']\in \hz$ with $d(z')\ge t-1$ such that $[z]=[z'][l]$.  
\end{enumerate}

\begin{enumerate}
\item[$\mathcal Z_{t,b,s}$:]There exists a set of cycles $Z$ in $\kz$ with $d(z)\ge t$  for all $z\in Z$ and $zz'=0$ for all $z,z'\in Z$, and such that for every $[v]\in \HH$ with $d(v)=s$ there exists $m\in \mathbb N$ and $z_i \in Z$, $[u_i]\in \HH$ with $d(z_i)\ge b$ for each $1\le i\le m$, such that $[v]=\sum_{i=1}^m[z_i][u_i]$. 
\end{enumerate}

We now  concentrate on the consequences of our results to the study of the Koszul property of $R$.  As recalled earlier, \cite[Corollary 1]{Ro} shows that the map $\Tor_i^R(\fm^2,k)\to \Tor_i^R(\fm,k)$ induced by the inclusion $\fm^2\subseteq \fm$ is zero for all $i\ge 0$ if and only if the Yoneda algebra $\Ext_R(k,k)$ is generated in degree $1$. Since $R$ is a standard graded $k$-algebra, the last statement is equivalent to the fact that $R$ is a Koszul algebra. 

We also consider a property that is stronger than Koszulness. As defined in \cite{IR}, a local (or graded) ring is said to be {\it absolutely Koszul} if the linearity defect of every finitely generated $R$-module is finite. While we refer to \cite{HI} for the original  definition of linearity defect, we mention that a module $M$ has finite linearity defect if and only if it has a syzygy $N$ whose associated graded module $\gr_{\fm}(N)$  has a linear resolution over the associated graded ring $\gr_{\fm}(R)$.  In the graded case, if $R$ is absolutely Koszul, then it is also Koszul, see \cite[Proposition 1.13]{HI}.

As defined  in the introduction, the {\it linear strand} of $\hh$ is the set of elements $[z]$ with $d(z)=1$. The {\it nonlinear strands} are  composed of those elements  with  $d(z)>1$.   We say that the nonlinear strands of $\hh$ are {\it generated by a set} $\ov Z$ with $\ov Z\subseteq \hh$ if the nonlinear strands are contained in the ideal generated by $\ov Z$ in $\hh$. If the nonlinear strands are generated by a subset $\ov Z$ of the linear strand, it follows that $\hh$ is generated by the linear strand as a $k$-algebra.

\begin{theorem}
\label{graded-thm}
Assume one of the following conditions holds: 
\begin{enumerate}[\quad\rm(1)]
\item  There exists an element $[l]$ of bidegree $(1,2)$ such that the nonlinear strands are generated by $[l]$, that is,  every element  in the nonlinear strands of $\hz$ is a multiple of $[l]$.  
\item $R_{\ges 3}=0$ and there exists a set of cycles $Z$  representing elements in the linear strand, with the property that $zz'=0$ for all $z, z'\in Z$, such that the set $\ov Z=\{[z]\mid z\in Z\}$ generates the nonlinear strand of $\hh$. 
\end{enumerate}
Then $R$ is Koszul. Moreover, $R$ is absolutely Koszul when {\rm (1)}  holds. 
\end{theorem}

\begin{proof}
If (1) holds, then $\mathcal P_{2,1}$ holds. By Corollary \ref{tor}, the induced maps $\Tor_*^R(\fm^2,k)\to \Tor_*^R(\fm,k)$ are zero. This implies that $R$ is Koszul. 

One of the consequences of the hypothesis in (1) is that there are no elements in bidegree $(1,i)$ with $i>2$. Consequently, the ideal $I$ is quadratic and $v(R)=2$. Let $L$ denote a preimage of $l$ in $\kz^Q$ and set $h=\dd(L)$. Note that $h\in I\smallsetminus \fn^3$. We can apply then Corollary \ref{P-R} with $a=1$ and $t=2$ to conclude that $R$ is a homomorphic image of a  quadratic hypersurface via a Golod homomorphism, hence $R$ is absolutely Koszul by \cite[Theorem 5.9]{HI}. 

If (2) holds, then $\mathcal Z_{1,1,2}$ holds. We apply then the graded version of Theorem  \ref{thm1} to conclude that the induced map $\Tor_i^R(\fm^2,k)\to \Tor_i^R(\fm,k)$ is zero, hence $R$ is Koszul. 
\end{proof}

There exist Koszul algebras that do not satisfy either of the conditions of the theorem, since, as noted  in \cite{Boocher-others}, the fact that $R$ is Koszul does not necessarily imply that the Koszul homology is generated by the linear strand.  On the other hand, the fact that $R$ is Koszul does impose conditions on the Koszul homology; see the introduction for more details on known results.  Of particular interest is the following reformulation of  Lemma \ref{r} from Section \ref{background_section}. The statement of this result was communicated to us orally by S.~Iyengar, who arrived at it in work with L.~Avramov and A.~Conca. 

\begin{proposition}\label{MHK}
If $R$ is Koszul, then the nonlinear strands of $\hh$ are contained in the set of matric Massey products MH$(\kz)$. 
\end{proposition}

\begin{proof}
If $R$ is Koszul, then the induced map $\Tor_*^R(\fm^2,k)\to \Tor_*^R(\fm,k)$ is zero, hence the canonical projection $R\to R/\fm^2$ is Golod, and thus small. Apply then Lemma \ref{r} to see that $\HH_{\ge 1}(\fm^2 \kz)\subseteq MH(\kz)$.
\end{proof}

In the next section, we discuss an example of a graded algebra  $R$ with $R_{\ges 3}=0$ for which the nonlinear strand is generated by the linear strand, but $R$ is not Koszul, see \ref{ring_not_Koszul}. Thus, the converse of Proposition \ref{MHK} does not hold. Stronger  hypotheses on the generation of Koszul homology such as the ones in our theorem  are thus needed in order to ensure $R$ is Koszul.

\section{Examples}
\label{examples}

We now proceed to give the relevant examples mentioned above. The computations here are done with the help of the Macaulay2 package \texttt{DGAlgebras} written by Frank Moore, but all computations can also be checked by hand (for example, see \ref{4socle} below). In this section, $k$ denotes a field of characteristic $0$. 

Let $Q=k[X,Y,Z,U]$, and $R=Q/\mathfrak{a}$ for an ideal $\mathfrak{a}$. Set $\mathfrak{m}=(x,y,z,u)$, with $x,y,z,u$ the images of $X,Y,Z,U$ in $R$, respectively, and $T_1,...,T_4$ the degree one variables of the Tate complex mapping to $x,y,z,u$, respectively.  As before, $\kz^R$ is the Koszul complex on $(x,y,z,u)$ and $\hz^R$ is its homology algebra. If $z\in \kz^R$ is a cycle, we denote by $[z]$ its homology class in $\hz^R$. Recall (e.g., \cite{Roo16}) that $\hz^R\cong \Tor^Q(R,k)$, and so the Betti table of $R$ over $Q$ gives an indication in which bidegrees elements of $\hz^R$ reside.

We start by applying Theorem \ref{graded-thm} to rings having the form $R=Q/\mathfrak{a}$ where $\a$ is an ideal generated by quadratic forms in $(X,Y,Z,U)$; such rings were studied by Roos in \cite{Roo16}. We verify the conditions of Theorem \ref{graded-thm} hold for 40+2=42 of the 104 rings in \cite[Tables A-D]{Roo16}, showing in particular that 40 of them are absolutely Koszul.

\begin{chunk}
\label{ring_one_gen}
For our first example, consider the ring $R$ which is case 66 in \cite[Table C]{Roo16}: 
\begin{align*}
R=k[X,Y,Z,U]/(XZ,Y^2,YU,Z^2,ZU,U^2).
\end{align*}
We claim that the nonlinear strand of $\hz^R$ is generated by a single element in bidegree $(1,2)$, i.e., it satisfies Theorem \ref{graded-thm}(1). As a result, we can conclude that this ring is absolutely Koszul, hence Koszul. To do this, we use the Macaulay2 package {\em DGAlgebras}.

In the following code, the set \texttt{C} is a list of cycles whose images (given by \texttt{G}) generate \texttt{H} ($=\hz^R$) as a $k$-algebra; we use the list of generators not in the linear strand, \texttt{P}, to define an ideal \texttt{I} that represents the nonlinear strand of \texttt{H}. Showing generation of the nonlinear strand by an element in bidegree $(1,2)$ now reduces to checking ideal containment.

\begin{quote}{\footnotesize\begin{verbatim}
needsPackage "DGAlgebras"
R=QQ[x,y,z,u]/ideal(x*z,y^2,y*u,z^2,z*u,u^2)
K=koszulComplexDGA(R)
C=getGenerators(K)
H=HH K
G=generators H
P={}; for n from 0 to length(G)-1 do {
    if (degree G_n)_0+1 !=(degree G_n)_1 then P=append(P,G_n)}
I=(ideal G)^2+(ideal P) --I is the ideal of nonlinear strands of H
m=0; for n from 1 to length(G) do { if degree X_n =={1,2} then m=m+1 else continue}
M=sum(m, j-> G_j)
N=sum(m, j-> C_j)
J=ideal(M)
isSubset(I,J) -- returns true if the nonlinear strands are generated by M
\end{verbatim}
}\end{quote}
In this example, we see that the homology class of the cycle $N=zT_1+(y+u)T_2+(z+u)T_3+uT_4$ generates the nonlinear strand of $\hz^R$, hence by Theorem \ref{graded-thm}(1), $R$ is absolutely Koszul.

Moreover, a similar argument shows that 40 of the rings from \cite[Tables A-D]{Roo16} have the nonlinear strand of their Koszul homology algebra generated by a single element in bidegree $(1,2)$, and hence are absolutely Koszul rings; these are cases 1-4, 8-10, 23, 25-28, 49, 50, 52, 53, 66-68, 70, 72, 75-83, and isotopes 46va, 66v5, 68v, 71v4, 72v2e, 75v2, 78v1, 78v2e, 78v3v, and 81va from \cite[Tables A-D]{Roo16}. Indeed, in all these cases we show that the homology class of the sum of generating cycles in bidegree $(1,2)$ found by \texttt{DGAlgebras} can be taken to be such a generator.  
\end{chunk}

The next example is a ring satisfying condition (2), but not condition (1), of Theorem \ref{graded-thm}.
\begin{chunk}
\label{ring_gen_by_set} %related to old \label{Koszulex} and \label{Koszulex-prop}
Let $R$ be the ring which is case 54 in \cite{Roo16}:
\begin{align*}
R=k[X,Y,Z,U]/(X^2,XZ,Y^2,Z^2,YU+ZU,U^2).
\end{align*}
The nonlinear strand of $\hz^R$ cannot be generated by a single element in bidegree $(1,2)$ as in the previous example, which can be seen by computing the Betti table for $R$ over $Q$:
\begin{verbatim}
            0 1  2  3 4
     total: 1 6 13 12 4
         0: 1 .  .  . .
         1: . 6  4  . .
         2: . .  9 12 4
\end{verbatim}
and observing there is no way for an element in bidegree $(4,6)$ to be a multiple of an element in bidegree $(1,2)$.

Next, we use the package \texttt{DGAlgebras} again to show that the nonlinear strand of $\hz^R$ is generated by a finite set of classes of cycles with trivial self-multiplication, that is, the ring $R$ satisfies condition (2) of Theorem \ref{graded-thm}. 
\begin{quote}{\footnotesize\begin{verbatim}
R=QQ[x,y,z,u]/ideal(x^2,x*z,y^2,z^2,y*u+z*u,u^2)
m=ideal vars R; m^3==0
betti res(ideal R)
K=koszulComplexDGA(R)
C=getGenerators(K)
H=HH K
G=generators H
P={}; for n from 0 to length(G)-1 do {
    if (degree G_n)_0+1 !=(degree G_n)_1 then P=append(P,G_n)}
I=(ideal G)^2+(ideal P) --I is the ideal of nonlinear strands of H
Cyc = {C_0,C_2,C_3,C_6,C_7}
Cls = {G_0,G_2,G_3,G_6,G_7}
for m from 0 to length(Cyc)-1 do {
    for n from 0 to length(Cyc)-1 do {
        if Cyc_m*Cyc_n==0 then TrivMult=true else {TrivMult=false; break} }}
print TrivMult --returns true if Cyc has trivial self-multiplication
J=ideal Cls
isSubset(I,J) --returns true if the nonlinear strands are generated by Cls
\end{verbatim}
}\end{quote}
In this example, the set \texttt{Cyc} is the desired generating set contained in the linear strand and its elements  correspond to the following cycles: 
$$\{xT_1,zT_3,zT_1,zT_1T_3,xT_1T_3\}.$$
The Macaulay2 code first checks $\m^3=0$, and then verifies that all the products of elements in \texttt{Cyc} are 0 and that the nonlinear strand \texttt{I} is contained in the ideal generated by \texttt{Cls} (the images in \texttt{H} of cycles in \texttt{Cyc}).
A similar argument shows that the ring in case 71 of \cite[Tables A-D]{Roo16} also satisfies Theorem \ref{graded-thm}(2); for that ring, one can show that the set \texttt{Cyc=}$\{C_0,C_1,C_2,C_7,C_8,C_9,C_{11},C_{15}\}$ is the desired generating set with trivial self-multiplication. 
\end{chunk}

\begin{remark}
Even among other rings with $\m^3=0$ in \cite{Roo16}, there are limitations to Theorem \ref{graded-thm}(2): For example, the ring in case 71v16 of \cite{Roo16} has no generating set satisfying this condition, despite being Koszul. Using \texttt{DGAlgebras} as above, we see that the ring
\begin{align*}
R=k[X,Y,Z,U]/(X^2,Y^2+Z^2,XY,YZ,ZU,XZ+U^2,XU),
\end{align*}
has $\m^3=0$ and has $H^R$ generated by $X_1,...,X_{17}$ such that $X_{13}^2\not=0$, hence the cycle corresponding to $X_{13}$ cannot be a part of any set with trivial self-multiplication. However, without $X_{13}$, we cannot generate the nonlinear strand:
\begin{quote}{\footnotesize\begin{verbatim}
R=QQ[x,y,z,u]/ideal(x^2,y^2+z^2,x*y,y*z,z*u,x*z+u^2,x*u)
G=generators HH koszulComplexDGA(R)
X_13*X_13==0 -- returns false
P={}; for n from 0 to length(G)-1 do {
    if (degree G_n)_0+1 !=(degree G_n)_1 then P=append(P,G_n)}
I=(ideal G)^2+(ideal P) --I is the ideal of nonlinear strands of H
isSubset(I,ideal delete(G_12,G)) -- returns false
\end{verbatim}
}\end{quote}
Hence no set satisfying the conditions of Theorem \ref{graded-thm}(2) can exist for this ring, but $R$ is Koszul by \cite[Main Theorem]{Roo16}.
\end{remark}

The next example shows that generation of $\hz^R$ by the linear strand alone cannot detect Koszulness of $R$. 
\begin{chunk}
\label{ring_not_Koszul} %related to old label{Roos-ex}
Consider the ring which is case 55 in \cite{Roo16}:
\begin{align*}
R=k[X,Y,Z,U]/(X^2+XY,XZ+YU,XU,Y^2,Z^2,ZU+U^2).
\end{align*}
The graded Betti table of $R$ over $Q$ is the same as for the ring in \ref{ring_gen_by_set}. Moreover, we see that $\hz^R$ is generated by the linear strand. This code can also be used to show the ring in \ref{ring_gen_by_set} is generated by its linear strand.
\begin{quote}{\footnotesize\begin{verbatim}
R=QQ[x,y,z,u]/ideal(x^2+x*y,x*z+y*u,x*u,y^2,z^2,z*u+u^2)
betti res(ideal R)
H=HH koszulComplexDGA(R)
G=generators H
for n from 0 to length(G)-1 list degree G_n
betti res(coker vars R, LengthLimit =>7)
\end{verbatim}
}\end{quote}
We see that $\hz^R$ has 6 generators in bidegree $(1,2)$ and 4 generators in bidegree $(2,3)$, all in the linear strand; further, the resolution of $k$ over $R$ is not linear, hence $R$ is not Koszul.
\end{chunk}

\begin{remark}
The Koszul homology algebras of the rings in \ref{ring_gen_by_set} and \ref{ring_not_Koszul} share the same Hilbert series and are both generated by the linear strand, yet one ring is Koszul and the other one is not. Thus,  generation by the linear strand and ``good" Hilbert series of the Koszul homology are not sufficient to decide  whether the ring is Koszul. The particularities of the generation of certain nonlinear strands in the Koszul homology seem to be relevant factors in detecting good homological behavior. 
\end{remark}

We now showcase the applicability of our results with a non-quadratic example: a non-compressed level algebra of socle degree $4$, with defining ideal generated in degree $3$.

\begin{chunk}
\label{4socle}
Let $Q=k[A,B,C,D]$, and let $I$ be the ideal 
$$(A^3, A^2C, A^2D, AC^2, B^3, B^2C, B^2D, BC^2, BD^2, C^2D, AB^2+CD^2, ABD-C^3, BCD+D^3),$$
and set $R=Q/I$. Write $\fm$ for the maximal homogeneous ideal of $R$ and write $a, b, c,d$ for the images of the variables.  
 As before we denote by $T_1, \ldots, T_4$ the variables of the Tate complex in degree one, mapping to $a,b,c,d$. 
One computes with Macaulay2 the Betti table of $R$ over $Q$: 

\begin{verbatim}
            0  1  2  3 4
     total: 1 13 22 12 2
         0: 1  .  .  . .
         1: .  .  .  . .
         2: . 13 19  5 .
         3: .  .  3  6 .
         4: .  .  .  1 2
\end{verbatim}

\begin{Proposition}
\label{4socle-prop}
The following hold: 
\begin{enumerate}[\quad\rm(1)]
\item The last strand of the Koszul homology $\hh$ is generated by one element in $\hh_{1,3}$. 
\item $R\to R/\fm^4$ is a Golod homomorphism and $R/\m^4$ is a Golod ring. 
\item $P^R_k(t)={\displaystyle \frac{(1+z)^3}{1-z-12z^2-10z^3-z^4+2z^5}}$.
\end{enumerate}
\end{Proposition}

To give a feel for how computations in $\hz^R$ are done by hand, and because the computations in this example are not too tedious, we give a proof that does not rely on the \texttt{DGAlgebras} package in Macaulay2. 

\begin{proof}
Using Macaulay2, we check that the reduced Gr\"obner basis of $I$ with respect to the lexicographic order is composed of the elements:
\begin{align*}
A^3, & A^2C, A^2D, AC^2, B^3, B^2C, B^2D, BC^2, BD^2, C^2D, AB^2+CD^2, \\
& ABD-C^3 BCD+D^3, BC^3, AD^3+C^4, D^4, CD^3, C^5.
\end{align*}
 Using this information, one sees that  $c^4$ and $acd^2$ generate the two-dimensional socle of $R$.

Note that the element $k_1 = (ac-bd)T_1 + c^2T_3\in \kz_{1,3}$ is a cycle. 
We will show that its class $[k_1]$ generates $\hh_{4,8}$ and $\hh_{3,7}$. 

\begin{Claim1}
$\hh_{4,8}$ is generated by the following two elements: 
$$
[c^4 T_1T_2T_3T_4],\  [acd^2 T_1T_2T_3T_4]
$$
This can be seen from the fact that $c^4, acd^2$ form a basis for the socle of $R$. 
\end{Claim1}

\begin{Claim2}
 $\hh_{3,7}$ is generated by the following element:
$$
[c^4T_1T_3T_4]
$$
\end{Claim2}
To verify this claim, note first that $c^4T_1T_3T_4$ is a cycle. Then, compute the module of boundaries $\fm^3 \dd(K_4)$  as being generated by  the classes of the following elements: 
$$ c^4 T_1T_2T_3, c^4 T_1T_2T_4, acd^2 T_1T_2T_4, acd^2 T_1T_3T_4, c^4 T_2T_3T_4, acd^2 T_2T_3T_4, acd^2 T_1T_2T_3 - c^4 T_1T_3T_4$$
and check that $[c^4T_1T_3T_4]\notin \fm^3 \dd(K_4)$. Since $\dim \hh_{3,7}=1$, this proves the claim.

Consider now the elements $k_2 = c^2T_1T_2T_4$ and $k_3 = (bc + d^2)T_2T_3T_4 - b^2T_1T_2T_4$, and $k_4=c^2T_1T_4$ which are also cycles. To verify that $[k_1]$ generates $\hh_{4,8}$ and $\hh_{3,7}$, we note the following relations:
\begin{align*}
k_1k_2&=c^4 T_1T_2T_3T_4\\
k_1k_3&=acd^2 T_1T_2T_3T_4\\
k_1k_4 &= -c^4T_1T_3T_4
\end{align*}

We conclude that the ring $R$ satisfies the graded version of condition $\mathcal Z_{2,2,4}$, as stated in Section \ref{graded}. 

Corollary \ref{formula}, together with a usage of Macaulay2 for computing the Poincar\'e series of $R/\fm^4$ over the polynomial ring, gives: 
$$
P^R_k(z)=\frac{(1+z)^4}{1-z(15z+30z^2+23z^3+7z^4)+2z^2(1+z)^4}=\frac{(1+z)^3}{1-z-12z^2-10z^3-z^4+2z^5}.
$$
\end{proof}
\end{chunk}

%\bibliographystyle{plain}  
%\bibliography{bibliographyKoszul}

\end{document}